\documentclass[12pt]{amsproc}

\usepackage[margin = 1.25in]{geometry}
\usepackage{amsmath,amsfonts,amssymb,amsthm}
\usepackage[usenames,dvipsnames]{xcolor}
\usepackage{graphicx}
\usepackage{tikz}
\usetikzlibrary{positioning}
\usepackage{tikz-cd}

\usepackage{hyperref}
\hypersetup{
    colorlinks,	%
    citecolor=blue,%
    filecolor=black,%
    linkcolor=red,%
    urlcolor=gray
}

\numberwithin{equation}{section}

	\newtheorem{theorem}[equation]{Theorem}
	\newtheorem{lemma}[equation]{Lemma}
	\newtheorem{corollary}[equation]{Corollary}
	\newtheorem{proposition}[equation]{Proposition}
	\newtheorem*{theorem*}{Theorem}

\theoremstyle{definition}
	\newtheorem{definition}[equation]{Definition}
	\newtheorem{example}[equation]{Example}
        \newtheorem{assumption}[equation]{Assumption}
	\newtheorem{notation}[equation]{Notation}

\theoremstyle{remark}
	\newtheorem{remark}[equation]{Remark}
	\newtheorem*{remark*}{Remark}

 { \begin{list}%
         {$\bullet$}%
         {\setlength{\labelwidth}{20pt}%
          \setlength{\leftmargin}{25pt}%
          \setlength{\topsep}{0pt}
          \setlength{\itemsep}{1.5ex}
          \setlength{\parsep}{0pt}}}%
 { \end{list} }

\newcommand{\R}{\mathbb{R}}
\newcommand{\Z}{\mathbb{Z}}
\newcommand{\N}{\mathbb{N}}
\newcommand{\Q}{\mathbb{Q}}

\newcommand{\cat}[1]{\mathbf{#1}}
\newcommand{\Ab}{\cat{Ab}}
\newcommand{\RMod}{\cat{Mod_R}}

\newcommand{\xto}{\xrightarrow}
\newcommand{\from}{\leftarrow}
\newcommand{\xfrom}{\xleftarrow}
\newcommand{\incl}{\hookrightarrow}

\newcommand{\isomto}{\stackrel{\isom}{\to}}

\newcommand{\isom}{\cong}
\newcommand{\symmdiff}{\bigtriangleup}

\newcommand\abs[1]{\lvert#1\rvert}
\newcommand\norm[1]{\left\lVert#1\right\rVert}

\newcommand{\eps}{\varepsilon}
\newcommand{\tensor}{\otimes}

\DeclareMathOperator{\coker}{coker}	
\DeclareMathOperator{\Dgm}{Dgm}

\DeclareMathOperator{\cost}{cost}
\DeclareMathOperator{\pd}{pd}

\renewcommand{\coprod}{\bigoplus}

\begin{document}

\title{Exact weights, path metrics, and algebraic Wasserstein distances}

\author{Peter Bubenik}
\address[Peter Bubenik]{Department of Mathematics, University of Florida}
\email{peter.bubenik@ufl.edu}

\author{Jonathan Scott}
\address[Jonathan Scott]{Department of Mathematics, Cleveland State University}
\email{j.a.scott3@csuohio.edu}

\author{Donald Stanley}
\address[Donald Stanley]{Department of Mathematics, University of Regina}
\email{donald.stanley@uregina.ca}


\begin{abstract}
  We use weights on objects in an abelian category to define what we call a path metric. We introduce three special classes of weight: those compatible with short exact sequences; those induced by their path metric; and those which bound  their path metric. We prove that these conditions are in fact equivalent, and call such weights exact. As a special case of a path metric, we obtain a distance for generalized persistence modules whose indexing category is a measure space. We use this distance to define Wasserstein distances, which coincide with the previously defined Wasserstein distances for one-parameter persistence modules. For one-parameter persistence modules, we also describe maps to and from an interval module, and we give a matrix reduction for monomorphisms and epimorphisms.
\end{abstract}

\maketitle
    
\section{Introduction} \label{sec:intro}

In nice cases, one-parameter persistence modules are isomorphic to a direct sum of interval modules~\cite{Crawley-Boevey:2015,Botnan:2018} and they have a combinatorial description called a persistence diagram~\cite{cseh:stability,Patel:2018}.
Persistence diagrams have a family of $L^p$ distances, for $1 \leq p \leq \infty$, called \emph{$p$-Wasserstein distances}~\cite{csehm:lipschitz}. For $p=\infty$, this distance is also called the \emph{bottleneck distance}~\cite{cseh:stability}.
These distances have a common generalization with Wasserstein distances for probability measures~\cite{Divol:2019,Bubenik:2020a}.
The bottleneck distance for one-parameter persistence modules has an equivalent linear-algebra formulation called \emph{interleaving distance}~\cite{ccsggo:interleaving,Lesnick:2011,bubenikScott:1,bauerLesnick,Harker:2019}
which has been extended to various generalized persistence modules~
~\cite{Morozov:2013,bdss:1,deSilva:2016,deSilva:2017,Blumberg:2017,bdss:2,Munch:2018,Botnan:2020a}. 
However, from the metric point of view, these distances, being $L^{\infty}$ distances, are rather weak. Saying that two persistence modules are close in $p$-Wasserstein distance for $p<\infty$ is much stronger, with $1$-Wasserstein distance giving the strongest notion of proximity.

We generalize the $1$-Wasserstein distance for one-parameter persistence modules to abelian categories. If these abelian categories satisfy some additional standard axioms we also obtain a generalization of the $p$-Wasserstein distances.

For an abelian category $\cat{A}$ a \emph{weight} assigns each object $A \in \cat{A}$ an associated weight $w(A) \in [0,\infty]$ such that $w(0)=0$ and if $A \isom B$ then $w(A)=w(B)$.
For example, for a field $K$ and the category of $K$-vector spaces, we have the weight $w(A)$ given by the dimension of $A$.
For another example, for a ring $R$ and the category of left $R$-modules, we have the weight $w(M) = \pd(M)+1$, where $\pd(M)$ denotes the projective dimension of $M$.
Say that a weight $w$ is \emph{exact} (Definition~\ref{def:exact-weight}) if for each short exact sequence $0 \to A \to B \to C \to 0$, $w(A) \leq w(B) + w(C)$, $w(B) \leq w(A) + w(C)$, and $w(C) \leq w(A) + w(B)$.
Both of the previous two examples of weights are exact.

Given $A, B \in \cat{A}$, a \emph{zigzag} from $A$ to $B$ consists of a sequence of morphisms $\gamma: A = A_0 \xto{\gamma_1} A_1 \xfrom{\gamma_2} A_2 \xto{\gamma_3} \cdots \xfrom{\gamma_n} A_n = B$ for some $n \geq 0$. Define the \emph{cost} of a zigzag by
\begin{equation*}
  \cost_{w}(\gamma) = \sum_{i=1}^n \left( w(\ker\gamma_i(p))\ +  w(\coker\gamma_i(p)) \right),
\end{equation*}
and let
$d_{w}(A,B) = \inf_{\gamma} \cost_{w}(\gamma)$,
where the infimum is taken over all zigzags between $A$ and $B$ (Definition~\ref{def:dw-morphism}).
We show (Lemma~\ref{lem:path-metric}) that $d_{w}$ is a metric (Definition~\ref{def:metric})
which we call the \emph{path metric}.

Given a metric $d$ on $\cat{A}$, there is a weight given by $\abs{d}(A) = d(A,0)$ (Definition~\ref{def:wd}).
  Therefore, given a weight $w$, we obtain a sequence of weights $w_1,w_2,w_3,\ldots$ with $w_1=w$ and $w_{n+1} = \abs{d_{w_n}}$ for $n \geq 1$.
  We prove that $w_1 \geq w_2 \geq w_3 \geq \cdots$ (Lemma~\ref{lem:bound-weight}).
  This sequence \emph{stabilizes} 
  if there exists an $n \geq 1$ such that $w_{n+1}=w_n$.
We call a weight $w$ \emph{stable} if $\abs{d_w} = w$ (Definition~\ref{def:stable-weight}).

We prove that any weight provides an upper bound for its path metric: $d_w(A,B) \leq w(A) + w(B)$ (Proposition~\ref{prop:upper-bound}).
  We say that a weight \emph{bounds its path metric} if in addition $\abs{w(A) - w(B)} \leq d_w(A,B)$ (Definition~\ref{def:lower-bound}).

We prove that the three seemingly unrelated conditions on weights we have introduced are in fact equivalent.
  \begin{theorem}[Theorem~\ref{thm:equivalent-conditions}]
    For a weight $w$ the following are equivalent:
    \begin{itemize}
    \item $w$ is exact;
    \item $w$ is stable; and
    \item $w$ bounds its path metric.
    \end{itemize}
  \end{theorem}

We also show  (Definition~\ref{def:strengthening}) that for each weight there is a canonical exact weight and that for each exact weight there is a canonical \emph{amplitude}, a strengthening of our notion of exact weight introduced by Giunti et al~\cite{Giunti:2021c} (Definition~\ref{def:amplitude}).

A \emph{persistence module} indexed by a small category $\cat{P}$ and valued in $\cat{A}$ is a functor from $\cat{P}$ to $\cat{A}$ and a \emph{morphism} of persistence modules is a natural transformation.
For example, consider $(\R^n,\leq)$ or $(\Z^n,\leq)$ with the coordinatewise partial order, viewed as a category.
The category of such persistence modules and their morphisms is an abelian category.

To define a path metric on this category of persistence modules, we use one additional ingredient. We assume that the underlying set $P$ of the small category $\cat{P}$ has a measure $\mu$. For example, consider $\R^n$ with the Lebesgue measure or $\Z^n$ with the counting measure.
  Then a persistence module $M$ has an associated weight defined by $W(M) = \mu(w(M)) =  \int_P w(M)\ d\mu$ (Definition~\ref{def:W-mu-w}).
If $w$ is exact or an amplitude then so is $W$ (Lemmas \ref{lem:muw-exact} and \ref{lem:muw-amplitude}).
  Using this weight we obtain the path metric $d_W = d_{\mu \circ w}$.

We prove that exact weights may be used to bound the path metric for persistence modules.

\begin{theorem}[Proposition~\ref{prop:upper-bound-pm} and Theorem~\ref{thm:lower-bound-pm}]
  If $w$ is an exact weight and $w(M)$ and $w(N)$ are $\mu$-integrable then
    \begin{equation*}
      \int_P\abs{w(M)-w(N)}\, d\mu \leq d_W(M,N) \leq \int_P (w(M)+w(N))\,d\mu.
    \end{equation*}
\end{theorem}

Now assume that the persistence modules have values in a Grothendieck category (Section~\ref{sec:grothendieck-category}) such as $\cat{Vect_K}$, the category of vector spaces over a field $K$, 
or $\RMod$, 
the category of left $R$-modules 
for some ring $R$,
and that they have a decomposition into a direct sum of persistence modules with local endomorphism rings (Section~\ref{sec:krsa}).
Given a metric $d$ and $p \in [1,\infty]$, we define the associated \emph{$p$-Wasserstein distance}
(Definition~\ref{def:Wp}), 
\begin{equation*}
  W_p(d)(M,N) = \inf \norm{\{d(M_a,N_a)\}_{a \in A}}_p,
\end{equation*}
where the infimum is taken over all isomorphisms
$M \isom \bigoplus_{a \in A} M_a$ and $N \isom \bigoplus_{a \in A} N_a$, where each $M_a$ and $N_a$
is either $0$ or
has a local endomorphism ring and is thus indecomposable (Lemma~\ref{lem:indecomposable}). 
We show that $W_p(d)$ is a metric (Proposition~\ref{prop:Wpd-metric}), which has the following universal property.

\begin{theorem}(Theorem~\ref{thm:universal})
  The metric $W_p(d)$ is the largest $p$-subadditive metric that
is bounded above by $d$ on indecomposables.
\end{theorem}

For one-parameter persistence modules we prove the following two isometry theorems.

\begin{theorem}[Theorem~\ref{thm:Wp-isometry}]
For persistence modules indexed by the integers or the real numbers with values in $\cat{Vect_K}$, $W_p(d_W)$ agrees with the $p$-Wasserstein distance of the corresponding persistence diagrams (Section \ref{sec:wasserstein-pers-mod}).
\end{theorem}

\begin{theorem}[Theorem~\ref{thm:metrics-same}]
For persistence modules indexed by the integers or the real numbers with values in $\cat{Vect_K}$, $W_1(d_W)$ agrees with the path metric $d_W$.
\end{theorem}

As part of the proof we show that
monomorphisms and epimorphisms of one-parameter persistence modules have the following representations which imply that there is an induced matching of interval modules.

\begin{theorem}[Theorem \ref{thm:matching-mono}]
  A monomorphism between persistence modules given by finite direct sums of interval modules can be represented by a matrix in which blocks corresponding to interval modules with the same right end are diagonal.
\end{theorem}

\begin{theorem}[Theorems \ref{thm:matching-epi}]
  An epimorphism between persistence modules given by finite direct sums of interval modules can be represented by a matrix in which blocks corresponding to interval modules with the same left end are diagonal.
\end{theorem}

We generalize the following well-known important elementary result for nonzero maps between persistence modules.

\begin{lemma}[Lemma~\ref{lem:interval-nonzero-map}]
Nonzero maps between interval modules may be visualized as follows.
\begin{center}
  \begin{tikzpicture}[scale=0.25]
    \draw[draw=none] (-1,-1) -- (21,5);
    \draw[thick] (0,0) -- (12,0);
    \draw[thick] (8,4) -- (20,4);
    \draw[thick,->] (10,3) -- (10,1);
  \end{tikzpicture}
\end{center}
\end{lemma}

\begin{theorem}[Theorem~\ref{thm:coker}]
Nonzero maps from an interval module to a finite direct sum of interval modules may be visualized as follows.
\begin{center}
  \begin{tikzpicture}[scale=0.25]
    \draw[draw=none] (-1,-6) -- (21,5);
    \draw[thick] (8,4) -- (20,4);
    \draw[thick,->] (8.25,3) -- (8.25,1);
    \draw[thick] (7,0) -- (9,0);
    \draw[thick] (6,-1) -- (10,-1);
    \draw[thick] (5,-2) -- (11,-2);
    \draw[thick] (4,-3) -- (12,-3);
    \draw[thick] (3,-4) -- (13,-4);
    \draw[thick] (2,-5) -- (14,-5);
  \end{tikzpicture}
\end{center}
\end{theorem}

\begin{theorem}[Theorem~\ref{thm:ker}]
Nonzero maps from a finite direct sum of interval modules to an interval module may be visualized as follows.
\begin{center}
  \begin{tikzpicture}[scale=0.25]
    \draw[draw=none] (-1,-1) -- (21,10);
    \draw[thick] (6,9) -- (18,9);
    \draw[thick] (7,8) -- (17,8);
    \draw[thick] (8,7) -- (16,7);
    \draw[thick] (9,6) -- (15,6);
    \draw[thick] (10,5) -- (14,5);
    \draw[thick] (11,4) -- (13,4);
    \draw[thick,->] (11.75,3) -- (11.75,1);
    \draw[thick] (0,0) -- (12,0);
  \end{tikzpicture}
\end{center}
\end{theorem}

\subsubsection*{Open questions}

We have not addressed algorithms for computing our path metric or our algebraic Wasserstein distance, under suitable finiteness conditions~\cite{lesnickWright:rivet,Miller:2019}. 
For example, is there an effective algorithm for computing the distance $d_{W}$ between two finitely-presented two-parameter persistence modules?
Furthermore, for particular applications in which generalized persistence modules arise, one may ask whether or not our distances are stable.

\subsubsection*{Related work}

Patel~\cite{Patel:2018} defines persistence diagrams for functors on $(\R,\leq)$ (which are obtained from functors on $(\N,\leq)$ by a left Kan extension) to essentially small symmetric monoidal categories with images and more generally to essentially small abelian categories.
In the latter case one can apply the tools developed here.
Note that our metric $d_{W}$ is similar in spirit to the construction of the Grothendieck group of an abelian category. 
Also note that the distances considered in \cite{Patel:2018} and the follow-up paper by McCleary and Patel~\cite{McCleary:2018} (interleaving distance, erosion distance, and bottleneck distance) are $L^{\infty}$ distances.
Elchesen and Memoli~\cite{Elchesen:2018} define a distance for zigzag persistence modules (the reflection distance) that is similar to our  metric $d_{W}$.
Related recent papers on the algebra of persistence modules include~\cite{Harrington:2017,Bauer:2020,Miller:2019,Miller:2020,Miller:2020c,BubenikMilicevic}.
The first author and Elchesen have also shown a universality result for Wasserstein distance for persistence diagrams~\cite{be:universality}.

Skraba and Turner~\cite{SkrabaTurner} have
independently defined an algebraic $p$-Wasserstein distance for pointwise-finite-dimensional one-parameter persistence modules and showed that for diagrams with finite total $p$-persistence it is isometric to the usual $p$-Wasserstein of the corresponding persistence diagrams. 

Scolamiero et al~\cite{Scolamiero:2017} define what they call a noise system for tame persistence modules indexed by $(\Q^r,\leq)$ and valued in $\cat{Vect_K}$ and use it to define a path metric.
Giunti, Nolan, Otter, and Waas~\cite{Giunti:2021c} have defined axioms for a weight on an abelian category which they call an amplitude. Their requirements are closely related to our conditions for an exact weight, but are more restrictive.
They observe that noise systems generalize to abelian categories and
they prove that an abelian category with an amplitude is equivalent to an abelian category with a noise system.
Thus, exact weights may be considered to be generalizations of amplitudes and noise systems.
For the path metric on noise systems and amplitudes,
it is sufficient to consider zigzags which are cospans (or spans)~\cite{Scolamiero:2017,Giunti:2021c}.

In Section~\ref{sec:matching} we show that for monomorphisms and epimorphisms of persistence modules there is an induced algebraic matching of interval modules.
Compare this with the induced combinatorial matchings of Bauer and Lesnick~\cite[Theorem 4.2]{bauerLesnick} and the related result by Skraba and Vejdemo Johansson~\cite[Remark 4.4]{bauerLesnick}.
A closely related result has been proved by Ezra Miller~\cite[Remark 9.24]{Miller:2020c}.
Miller's result holds in greater generality, though in our case his result is perhaps slightly weaker or at least less explicit.
Our proof is elementary, using a matrix reduction argument.

\subsubsection*{Outline of the paper}

Section~\ref{sec:background} consists of background material.
In Section~\ref{sec:weight-path-metric} we define weights and path metrics and study exact weights and their properties.
In Section~\ref{sec:metric-from-measure} we define metrics for generalized persistence modules indexed by a measure space and consider some of their properties.
In Section~\ref{sec:wasserstein} we define Wasserstein distances for persistence modules with values in a Grothendieck category, prove that it extends the usual definition, and establish a universal property.
In Section~\ref{sec:agreement}
we show that for one-parameter persistence modules our algebraic $1$-Wasserstein distance agrees with the path metric.
We also prove structure theorems for maps into and out of an interval module and show that monomorphisms and epimorphisms of persistence modules
can be represented by matrices whose form induces a matching of interval modules.
Finally, in Section~\ref{sec:applications}, we apply our metrics to three examples of two-parameter persistence modules and a pair of zigzag persistence modules.
 
\section{Background} \label{sec:background}

In this section we give background material that will be used later. 

\subsection{Additive categories}
\label{sec:additive-category}

A \emph{zero object} in a category is an object $0$ such that for every object $X$ there are unique morphisms $0 \to X$ and $X \to 0$.
In a category with a  zero object, for any two objects $A,B$ there is a unique \emph{zero morphism} given by the composition $A \to 0 \to B$.
An \emph{additive category} is one that is enriched in abelian groups (i.e. hom sets are abelian groups, and composition of morphisms is biadditive) and that has all finite products and a \emph{zero object}.

Let $\cat{A}$ be an additive category.
We say that $X$ is the \emph{direct sum} of $Y$ and $Z$ in $\cat{A}$ if there are morphisms $i : Y \rightarrow X$, $j : Z \rightarrow X$, $p : X \rightarrow Y$, and $q : X \rightarrow Z$ such that $ip + jq = 1_{X}$, $pi = 1_{Y}$, and $qj = 1_{Z}$.
Thus $p$ and $q$ are epimorphisms,
$i$ and $j$ are monomorphisms, and we consider $Y$ and $Z$ to be subobjects of $X$.
We write $X \cong Y \oplus Z$.
One can show that $qi = 0$ and $pj = 0$, from which it is easy to deduce that $i$ and $j$ determine an isomorphism $X \cong Y \amalg Z$, and that $p$ and $q$ determine an isomorphism $X \cong Y \times Z$.
An object $X \in \cat{A}$ is \emph{indecomposable} if $X \cong Y \oplus Z$ implies that either $Y$ or $Z$ is $0$.
See Krause~\cite{Krause:2015} for more details.

In an additive category $\cat{A}$, the \emph{kernel} of a morphism $f:A \to B$, if it exists, is the equalizer of $f$ and the zero morphism between $A$ and $B$.
Dually, the \emph{cokernel} of $f$, if it exists, is the coequalizer of $f$ and the zero morphism.

\subsection{Abelian categories}
\label{sec:abelian-category}

  An additive category is \emph{abelian} if it has all kernels and cokernels, and if for every $f : M \rightarrow N$, the induced morphism $\bar{f}$ in the natural factorization,
\[
	\begin{tikzcd}[row sep=tiny]
		\ker{f} \arrow[r,"j"]
			& M \arrow[r,"f"]
				\arrow[d]
			& N \arrow[r,"q"] 
			& \coker{f} \\
			& \coker{j}
				\arrow[r,"\bar{f}"]
			& \ker{q}
				\arrow[u]
	\end{tikzcd}
\]
is an isomorphism.
Note that $\ker q$ is called the image of $f$ and $\coker j$ is called the coimage of $f$.

Let $R$ be a commutative ring (with identity).
Then the category $\RMod$ of $R$-modules and $R$-module homomorphisms is an abelian category. As a special case, 
  let $K$ be a field.
  The category $\cat{Vect_K}$ of vector spaces over $K$ and $K$-linear maps is an abelian category.
  If $\cat{A}$ is an abelian category and $\cat{D}$ is a small category then the category $\cat{A}^{\cat{D}}$, of functors from $\cat{D}$ to $\cat{A}$ and natural transformations, is an abelian category.

\subsection{Grothendieck categories}
\label{sec:grothendieck-category}

An \emph{AB5 category} is an abelian category with all coproducts (and hence all colimits) in which filtered colimits of exact sequences are exact.
A \emph{Grothendieck category} is an AB5 category which has a generator (i.e. separator).

For example,
for any unital ring $R$, the category $\RMod$ of left $R$-modules and $R$-module homomorphisms is a Grothendieck category.
This includes the cases $\cat{Vect_K}$ (where $R$ is a field $K$) and $\Ab$ the category of abelian groups and group homomorphisms (where $R=\Z$).
Let $\cat{P}$ be a small category. For any Grothendieck category $\cat{A}$, the category $\cat{A}^{\cat{P}}$ is a Grothendieck category. In particular, $\cat{Vect_K}^{\cat{P}}$ is a Grothendieck category.

Let $\cat{A}$ be a Grothendieck category.
For an arbitrary set $A$ and a collection of objects $\{M_a\}_{a \in A}$ in $\cat{A}$, by definition we have the direct sum (i.e. coproduct) $\bigoplus_{a \in A}M_a$, and canonical maps $i_a:M_a \to \bigoplus_{a \in A}M_a$ for all $a \in A$.
It follows from the Gabriel-Popescu Theorem that $\cat{A}$ also has all limits~\cite[Chapter X]{Stenstrom:1975}, and thus products, in particular.
Therefore, we have the product $\prod_{a \in A} M_a$.
For $a,b \in A$ define $\tau_{a,b}:M_a \to M_b$ to be the identity on $M_a$ if $a=b$ and to be the zero map otherwise. For $b \in A$ the maps $\tau_{a,b}$ induce a canonical projection map $p_b: \bigoplus_{a \in A}M_a \to M_b$.
These maps induce a canonical map $\bigoplus_{a \in A}M_a \to \prod_{a \in A}M_a$.

\subsection{Krull-Remak-Schmidt-Azumaya Theorem}
\label{sec:krsa}

An element $r$ in a ring $R$ is a \emph{nonunit} if $Rr \neq R$ and $rR \neq R$.
A \emph{local ring} is a ring in which the sum of two nonunits is a nonunit. 

\begin{lemma}
  \label{lem:indecomposable}
  Let $\cat{A}$ be an abelian category. 
If $M \in \cat{A}$ has a local endomorphism ring, then $M$ is indecomposable.
\end{lemma}

\begin{proof}
  Assume $M \isom M_1 \oplus M_2$, with corresponding maps $i_1, p_1, i_2, p_2$. Then $i_1p_1$ and $i_2p_2$ are nonunits but their sum is not.
\end{proof}

\begin{theorem}[Krull-Remak-Schmidt-Azumaya Theorem]\cite[Section 6.7]{Bucur:1968},
\cite[Section 4.8]{Pareigis:1970},
\cite[Section 5.1]{Popescu:1973}.
 \label{thm:krsa}
  Let $\cat{A}$ be an AB5 category and $M \in \cat{A}$. If
  \begin{equation*}
    M \isom \displaystyle\coprod_{i \in I} A_i \isom \coprod_{j \in J} B_j,
  \end{equation*}
where each $A_i$ and $B_j$ has a local endomorphism ring, then there is a bijection $\varphi: I \to J$ such that for all $i \in I$, $A_i \isom B_{\varphi(i)}$.
\end{theorem}

\begin{definition} \label{def:subcat}
  For a Grothendieck category $\cat{A}$, let $\cat{A_{\ell}}$ denote the full additive subcategory of $\cat{A}$ whose objects are those objects of $\cat{A}$ that are isomorphic to a direct sum of objects with a local endomorphism ring.
\end{definition}

\subsection{Persistence modules}
\label{sec:pers-mod}


Let $\cat{P}$ be a small category and let $\cat{A}$ be an abelian category.
Functors $M: \cat{P} \to \cat{A}$ are called \emph{persistence modules} indexed by $\cat{P}$ with values in $\cat{A}$. Natural transformations of such functors are called \emph{morphisms of persistence modules}.
Of particular interest are the cases that $\cat{A}$ is $\RMod$ or its special case $\cat{Vect_K}$.
Let $P$ denote the set of objects of $\cat{P}$.
For a persistence module $M: \cat{P} \to \cat{Vect_K}$ 
the \emph{dimension vector} or \emph{Hilbert function} for $M$ is the function $\dim M: P \to [0,\infty]$ given by $p \mapsto \dim M(p)$.

Among persistence modules with values in $\cat{Vect_K}$, of greatest interest is the case where $P \subseteq \R^d$ for some $d$ and the morphisms are given by the coordinate-wise/product partial order: $(x_1,\ldots,x_d) \leq (y_1,\ldots,y_d)$ iff $x_i \leq y_i$ for all $1 \leq i \leq d$.
When $d \geq 2$ these are called \emph{multi-parameter persistence modules} and when $d=1$ these are called \emph{one-parameter persistence modules} or just \emph{persistence modules}.

\begin{definition} \label{def:interval-module}
  Let $P$ be a poset. A subset $C \subseteq P$ is \emph{convex} if for all $p \leq q \leq r$ with $p,r \in C$, we have $q \in C$.
    A subset $C \subseteq P$ is \emph{connected} if for each $p,q \in C$ there is a sequence $p = p_0, p_1,\ldots, p_n=q$ in $C$ such that for each $1 \leq j \leq n$, either $p_{j-1} \leq p_j$ or $p_j \leq p_{j-1}$.
    An \emph{interval} in $P$ is a convex connected subset. Note that if $P$ is totally ordered then an interval is just a convex subset.
  Let $I$ be an interval in $P$. Define a persistence module $M$ indexed by $P$ with values in $\cat{Vect_K}$ as follows. For each $p \in P$, let $M(p) = K$ if $p \in I$ and $M(p)= 0$ if $p \not\in I$. For $p \leq q$ with $p,q \in I$, let $M(p \leq q)$ be the identity map on $K$. All other maps $M(p \leq q)$ are zero, since either the domain or codomain is zero. Call $M$ an \emph{interval module} and it is convenient to abuse notation and denote $M$ by $I$.
\end{definition}

\begin{lemma} \label{lem:interval-module}
  Each interval module has a local endomorphism ring and is thus indecomposable. 
\end{lemma}

\begin{proof}
  The endomorphism ring of an interval module, which by definition has values in $\cat{Vect_K}$, is isomorphic to $K$.
\end{proof}

\subsection{$p$-Norms}
\label{sec:p-norm}

It is customary to restrict $p$-norms to those elements for which they have a finite value; we will not do so.
Let $x = \{x_a\}_{a \in A}$, where each $x_a \in [0,\infty]$.
Then for $1\leq p < \infty$, let
$\norm{x}_p = ( \sum_{a \in A} \abs{x_a}^p )^{\frac{1}{p}}$
and 
$\norm{x}_{\infty} = \sup_{a \in A} \abs{x_a}$.

\begin{lemma}
  Let $A$ and $B$ be disjoint indexing sets.
  Let $x = \{x_a\}_{a \in A}$, $y=\{x_b\}_{b \in B}$ and $z=\{x_c\}_{c \in A \cup B}$. 
Then for $1 \leq p \leq \infty$,
$\norm{ \left( \norm{x}_p, \norm{y}_p \right) }_p = \norm{z}_p$.
\end{lemma}


\subsection{Persistence diagrams and their Wasserstein distances}
\label{sec:wasserstein-pers-mod}

Let $P \subseteq \R$, where $\R$ is given the usual total order.
  For an interval $I$ in $P$, let $P_{>I} = \{p \in P \ | \ \forall x \in I, x < p\}$.
For an interval module $I$ indexed by $P$, let $x(I) =
(\inf I, \inf P_{>I})
\in [-\infty,\infty]^2$, where $\inf \emptyset = \infty$.
For $x,y \in [-\infty,\infty]^2$,
let $d(x,y) = \norm{x-y}_1$.  
Let $\Delta \subset [-\infty,\infty]^{2}$ denote the diagonal, $\{(x,x) \ | \ -\infty \leq x \leq \infty\}$ and for
$x \in [-\infty,\infty]^2$, let $d(x,\Delta) := \inf_{y \in \Delta} d(x,y)$.
By a \emph{matching} between index sets $A$ and $B$, we mean an injection $\varphi : C \rightarrow B$, where $C \subset A$.

Let $P \subset \R$ and let $M$ be a persistence module indexed by $P$ with values in $\cat{Vect_K}$. Assume that $M \isom \coprod_{j \in J} I_j$ where each $I_j$ is an interval module.
By Lemma~\ref{lem:interval-module} and Theorem~\ref{thm:krsa}, there is a well-defined multiset $\Dgm M := \{x(I_{j})\}_{j \in J}$, called the \emph{persistence diagram} of $M$.

\begin{definition} \label{def:Wpq}
Let $1 \leq p \leq \infty$.
Let $M,N$ be persistence modules indexed by $P$ with values in $\cat{Vect_K}$ that have persistence diagrams $\Dgm M = \{x_a\}_{a \in A}$ and $\Dgm N = \{x'_b\}_{b \in B}$.
Define
  \begin{multline*}
    W_{p}(M,N) =\\ \inf_{\varphi : C \rightarrow B} \norm{ \left(
        \norm{ \left\{ d(x_c,x'_{\varphi(c)}) \right\}_{c \in C} }_p,
        \norm{ \left\{ d(x_a,\Delta) \right\}_{a \in A-C} }_p,
        \norm{ \left\{ d(\Delta,x'_b) \right\}_{b \in B-\varphi(C)} }_p \right) }_p,
  \end{multline*}
  where the infimum is over all matchings $\varphi$ between the index sets $A$ and $B$.
  Call this the \emph{$p$-Wasserstein distance} between the persistence modules $M$ and $N$.
\end{definition}

We alert the reader that in~\cite{csehm:lipschitz}, the Wasserstein distance uses the $\infty$-norm to measure distances in $\R^2$. We  use the $1$-norm.

\subsection{Zigzags of morphisms} \label{sec:zigzag}

Let $\cat{A}$ be a category.
Let $M, N \in \cat{A}$.
A \emph{zigzag} of morphisms from $M$ to $N$ is a finite collection of morphisms in $\cat{A}$ of the form $M = M_0 \xto{f_1} M_1 \xfrom{f_2} M_2 \xto{f_3} \cdots \xfrom{f_n} M_n = N$.
The number $n \geq 0$ is called the \emph{length} of the zigzag.
Note that by inserting identity maps, we can allow the morphisms to point in either direction.

\subsection{Symmetric Lawvere metric}
\label{sec:lawvere}

A \emph{symmetric Lawvere metric}
is a class $\mathcal{C}$ together with a function $d$ that assigns to any pair $M,N \in \mathcal{C}$ a number $d(M,N) \in [0,\infty]$ such that for all $M \in \mathcal{C}$, $d(M,M) = 0$, for all $M,N \in \mathcal{C}$, $d(M,N) = d(N,M)$, and for all $M,N,P \in \mathcal{C}$, $d(M,P) \leq d(M,N) + d(N,P)$.
This definition relaxes the usual definition of a metric in three ways: it is allowed to take on the value $\infty$; $d(M,N)=0$ does not imply that $M=N$; and the class $\mathcal{C}$ is not required to be a set.

\section{Weights and path metrics} \label{sec:weight-path-metric}

In this section we use weights on morphisms in a category or weights on objects in an additive category to define a distance that we call a path metric.

\subsection{Weights and metrics on categories}
\label{sec:metric}

In this section we define weight,  give examples of weights, define a metric on a category, and give an elementary property of such a metric.

\begin{definition} \label{def:weight}
  A \emph{weight}, $w$, on a class $\mathcal{A}$ assigns $w(a) \in [0,\infty]$ to each $a \in \mathcal{A}$.
  A weight on a category is a weight on the class of all objects of the category.
\end{definition}

\begin{example} \label{ex:amplitudes}
For any category we have the \emph{zero weight} that assigns each object the weight $0$.
For any additive category we the \emph{one weight} that assigns all nonzero objects weight $1$ and the zero object weight $0$.
For an abelian category let $\mathcal{S}$ be the class of simple objects, whose only subobjects are $0$ and themselves, together with $0$.
Define a weight on $\mathcal{S}$, called the \emph{simple weight}, by $w(0)=0$ and $w(S)=1$ for all other $S \in \mathcal{S}$.
For a field $K$ and the category $\cat{Vect_K}$ of $K$-vector spaces, we have a weight given by the dimension of the vector space. Call this the \emph{dimension weight}.
More generally, if $R$ is an integral domain, then for the category $\cat{Mod_R}$ of $R$-modules, we have a weight given by the rank of a module $M$, which equals the dimension of $M \tensor_R K$ where $K$ is the field of fractions of $R$. 
Call this the \emph{rank weight}.
\end{example}

\begin{definition} \label{def:metric}
  Let $\mathcal{C}$ be a class of objects in a category $\cat{C}$. We define a \emph{metric} on $\mathcal{C}$ to be a symmetric Lawvere metric with the additional property that if $M,N \in \mathcal{C}$ with $M \isom N$ then $d(M,N) = 0$. A \emph{metric on a category} $\cat{C}$ is a metric on the class of all objects in $\cat{C}$.
\end{definition}
Our definition does allow non-isomorphic objects $M$ and $N$ to have $d(M,N) = 0$.
%
  Let $M,M',N,N' \in \mathcal{C}$ with $M \isom M'$ and $N \isom N'$.
  It follows from the triangle inequality that $d(M,N) = d(M',N')$.


\subsection{Path metric from a weight}

We use a weight on a class of morphisms in a category to define a metric for that category which we will call the path metric.
  As a special case we use a weight on a class of objects in an additive category to define a metric on that category.

  Let $\cat{C}$ be a category together with a class, $\mathcal{M}$, of morphisms in $\cat{C}$ and a weight $w$ on $\mathcal{M}$.

\begin{definition} \label{def:dw-morphism}
    Let $\gamma$ be a zigzag in $\cat{C}$ in which each morphism in the zigzag is in $\mathcal{M}$. Define the \emph{cost} of $\gamma$, denoted $\cost_w(\gamma)$, to be the sum of the weights of the morphisms in the zigzag.
    As a special case, the cost of the zigzag of length $0$ is $0$.
Let $A,B \in \cat{C}$. Define the \emph{path distance} by
$d_w(A,B) = \inf_{\gamma} \cost_w(\gamma)$, where the infimum is taken over all zigzags from $A$ to $B$ such that each morphism in the zigzag is in $\mathcal{M}$. If there are no such zigzags then let $d_w(A,B) = \infty$.
\end{definition}

\begin{lemma} \label{lem:path-metric}
  The path distance $d_w$ is a symmetric Lawvere metric  on $\cat{C}$ (Section~\ref{sec:lawvere}).
  If $\mathcal{M}$ includes all isomorphisms in $\cat{C}$ and the weight of each isomorphism is $0$, then $d_w$ is a metric on $\cat{C}$ (Section~\ref{sec:metric}), which we call the \emph{path metric}.
\end{lemma}

\begin{proof}
First, for any object $A$, $d_{w}(A,A)=0$ since there is a zigzag of length $0$ from $A$ to $A$, whose cost, by definition is $0$.
Next, $d_{w}(A,C) \leq d_{w}(A,B) + d_{w}(B,C)$ since we may concatenate a zigzag from $A$ to $B$ with a zigzag from $B$ to $C$ to obtain a zigzag from $A$ to $C$ whose cost is the sum of the costs of the two zigzags.
Furthermore, $d_{w}(A,B) = d_{w}(B,A)$ since every zigzag has a reverse zigzag with the same cost.

For the second statement, for isomorphic objects $A,B$, consider 
the zigzag of length $1$ given by
$f:A \isomto B$,
which has cost $0$. Thus $d_w(A,B) = 0$.
\end{proof}

\begin{assumption} \label{ass:O-w}
  Let $\cat{A}$ be an additive category. Let $\mathcal{O}$ be a class of objects in $\cat{A}$. We will always assume that such a class contains $0$ and that if $A \isom B$ and $A \in \mathcal{O}$ then $B \in \mathcal{O}$. Let
$w$ be a weight on $\mathcal{O}$.
We will always assume that $w(0) = 0$ and if $A \isom B$ and $A,B \in \mathcal{O}$ then $w(B) = w(A)$.
\end{assumption}

\begin{definition} \label{def:dw-object}
Let $\mathcal{O}$ be a class of objects in an additive category $\cat{A}$ and let $w$ a weight on $\mathcal{O}$. See Assumption~\ref{ass:O-w}. Let $\mathcal{M}$ be the class of morphisms in $\cat{A}$ whose kernel and cokernel and both are in $\mathcal{O}$. Define a weight on $\mathcal{M}$, which we also denote $w$, by $w(f) = w(\ker f) + w(\coker f)$.
  Note that it follows that $\mathcal{M}$ contains all isomorphisms and that these have weight $0$. 
Applying Definition~\ref{def:dw-morphism}, with zigzags of morphisms whose kernel and cokernel are in $\mathcal{O}$, we obtain a path distance $d_w$ on $\cat{A}$.
By Lemma~\ref{lem:path-metric}, $d_w$ is a metric, which call the path metric.
\end{definition}

Since any morphism in an abelian category factors through its image (Section~\ref{sec:abelian-category}) we have the following.

\begin{lemma} \label{lem:zigzag-zero}
  Assume that $\cat{A}$ is an abelian category. In Definition~\ref{def:dw-object}, if we restrict $\mathcal{M}$ to morphisms having either zero kernel or zero cokernel then we obtain the same path metric.
\end{lemma}

\subsection{Exact weights and amplitudes}

In this section we consider weights compatible with short exact sequences.

Let $\cat{A}$ be an abelian category (or more generally a (Quillen) exact category) together with a class of objects $\mathcal{O}$ containing $0$ and a weight $w$ on $\mathcal{O}$ (see Assumption~\ref{ass:O-w}).

\begin{definition} \label{def:exact-weight}
    Say that the weight $w$ on $\mathcal{O}$ is \emph{exact} if for each short exact sequence $0 \to A \to B \to C \to 0$ in $\cat{A}$ with $A,B,C \in \mathcal{O}$, $w(A) \leq w(B) + w(C)$, $w(B) \leq w(A) + w(C)$, and $w(C) \leq w(A) + w(B)$.
\end{definition}

The following are examples of exact weights.

\begin{example} \label{ex:nonamplitude}
  Let $\cat{vect_K}$ be the category of finite-dimensional vector spaces over $K$ and $K$-linear maps. For $V \in \cat{vect_K}$, let $w(V) = 0$ if $V=0$, otherwise $w(V) = 1$ if $\dim(V)$ is even and $w(V) = 2$ if $\dim(V)$ is odd.
  Then $w$ is an exact weight on $\cat{vect_K}$.
\end{example}

\begin{example} \label{ex:proj-dim}
  Let $\cat{Mod_R}$ denote the category of right (or left) $R$-modules over a ring $R$.
  For $A \in \cat{Mod_R}$, let $\pd(A)$ denote the projective dimension of $A$.
  We claim that $w(A) = \pd(A)+1$ is an exact weight.
  Consider a short exact sequence $0 \to A \to B \to C \to 0$ in $\cat{Mod_R}$.
  Using the characterization of projective dimension using ext groups and the long exact sequence of ext groups, one obtains
  $\pd(A) \leq \max(\pd(B),\pd(C))$,
  $\pd(B) \leq \max(\pd(A),\pd(C))$, and $\pd(C) \leq 1 + \max(\pd(A), \pd(B))$.
  It follows that $w$ is exact.
  Similarly, if we replace projective dimension with injective dimension or flat dimension, we also obtain an exact weight.
\end{example}

Giunti et al~\cite{Giunti:2021c} consider a stronger notion of exact weight on an abelian category which they call amplitude. We generalize their definition slightly to weights on $\mathcal{O}$.

\begin{definition} \label{def:amplitude}
    Say that the weight $w$ on $\mathcal{O}$ an \emph{amplitude} if for each short exact sequence $0 \to A \to B \to C \to 0$ in $\cat{A}$, if $A,B \in \mathcal{O}$ then $\alpha(A) \leq \alpha(B)$, if $B,C \in \mathcal{O}$ then $\alpha(C) \leq \alpha(B)$, and if $A,B,C \in \mathcal{O}$ then $\alpha(B) \leq \alpha(A) + \alpha(C)$. If, in addition, for each short exact sequence $0 \to A \to B \to C \to 0$ with $A,B,C \in \mathcal{O}$, $\alpha(B) = \alpha(A) + \alpha(C)$ then the amplitude is called \emph{additive}. 
\end{definition}

\begin{example}
  The zero weight on an abelian category is an additive amplitude.
  The one weight on an abelian category is a non-additive amplitude.
  Since any short exact sequence of vector spaces splits, the dimension weight is an additive amplitude.
  Since localization is an exact functor, the rank weight is also an additive amplitude.
  For many other examples of amplitude, see \cite{Giunti:2021c}.
  The exact weights in Examples \ref{ex:nonamplitude} and \ref{ex:proj-dim} are not amplitudes.
  The simple weight extends to an additive amplitude on the class of semisimple objects. In the case of $\cat{Vect_K}$ this produces the dimension weight.
\end{example}

\begin{example} 
For the one weight $w$ on an abelian category $\cat{A}$, the path metric $d_w$ satisfies the following. For $A,B \in \cat{A}$, $d_w(A,B) = 0$ iff $A \isom B$, $d_w(A,B) = 1$ iff $A \not\isom B$ and there exists either an injection or a surjection between $A$ and $B$, and otherwise $d_w(A,B) = 2$.
\end{example}

\subsection{How to obtain weights with stronger properties}

  We will show that each weight has a canonical associated exact weight and the each exact weight has a canonical associated amplitude.
Let $\cat{A}$ be an abelian category with a class of objects $\mathcal{O}$ including $0$.
  First, we need the following lemma.

\begin{lemma}
  Let $\{\alpha_j\}_{j \in J}$ be a set of amplitudes on $\mathcal{O}$. For $A \in \mathcal{O}$, let $\alpha(A) = \sup_{j \in J} \alpha_j(A)$.
  Then $\alpha$ is an amplitude on $\mathcal{O}$.
\end{lemma}

\begin{proof}
    To start, observe that $\alpha(0) = 0$. Next consider a short exact sequence $0 \to A \to B \to C \to 0$. Assume $A,B,C \in \mathcal{O}$.
    Let $\eps > 0$. Then by definition there is a $j \in J$ such that $\alpha_j(B) > \alpha(B) - \eps$. It follows by definition and by assumption that $\alpha(A) + \alpha(C) \geq \alpha_j(A) + \alpha_j(C) \geq \alpha_j(B) > \alpha(B) - \eps$. Therefore $\alpha(A) + \alpha(C) \geq \alpha(B)$. A similar argument shows that if $A,B \in \mathcal{O}$ then $\alpha(A) \leq \alpha(B)$ and that if $B,C \in \mathcal{O}$ then $\alpha(C) \leq \alpha(B)$.
\end{proof}

Similarly, we have the following.

\begin{lemma}
  Let $\{\alpha_j\}_{j \in J}$ be a set of exact weights on $\mathcal{O}$. For $A \in \mathcal{O}$, let $\alpha(A) = \sup_{j \in J} \alpha_j(A)$.
  Then $\alpha$ is an exact weight on $\mathcal{O}$.
\end{lemma}

For two weights $w,w'$ on $\mathcal{O}$ say that $w \leq w'$ iff for all $A \in \mathcal{O}$, $w(A) \leq w'(A)$.

\begin{definition} \label{def:strengthening}
    Let $w$ be a weight on $\mathcal{O}$. We define the associated exact weight $\underline{w}$ on $\mathcal{O}$ to the supremum of the exact weights upper-bounded by $w$. Note that this set of exact weights is nonempty because of the zero weight.
    We define the associated amplitude $\alpha_w$ on $\mathcal{O}$ to be the supremum of the amplitudes upper-bounded by $w$.
\end{definition}

Note that if $w$ is an exact weight then $\underline{w} = w$ and if
  $w$ is an amplitude, then $\alpha_w = w$.

\begin{example}
  For the exact weight $w$ in Example~\ref{ex:nonamplitude}, the associated amplitude is
  the one weight (Example~\ref{ex:amplitudes}).
  For the exact weight $w$ in Example~\ref{ex:proj-dim}, if $\cat{A}$ has enough projectives, then for each $A \in \cat{A}$ there is a surjection $P \xto{f} A$ with $P$ projective. From the short exact sequence $0 \to ker f \to P \to A \to 0$, we obtain $\alpha_w(A) \leq \alpha_w(P) \leq w(P) = 1$.
  Thus the associated amplitude is also the one weight.
\end{example}

\subsection{Weight from a metric and stable weights}

In this section, we use a metric on a category to define a weight on that category. Let $\cat{A}$ be a category together with a class of objects $\mathcal{O}$ in $\cat{A}$ (see Assumption~\ref{ass:O-w}).

\begin{definition} \label{def:wd}
  Let $d$ be a metric on $\cat{A}$ (Section~\ref{sec:metric}).
  For $A \in \cat{A}$ define $\abs{d}(A) = d(A,0)$.
  Then $\abs{d}(0) = 0$
  and if $A \isom B$ then $\abs{d}(A) = \abs{d}(B)$.
  Let $\abs{d}_{\mathcal{O}}$ denote the restriction of $\abs{d}$ to $\mathcal{O}$.
  Then $\abs{d}_{\mathcal{O}}$ is a weight on $\mathcal{O}$ (Assumption~\ref{ass:O-w}).
\end{definition}

From a weight we obtain a path metric and from this path metric we obtain a weight.

\begin{lemma} \label{lem:bound-weight}
  For a weight $w$ on $\mathcal{O}$, $\abs{d_w}_{\mathcal{O}} \leq w$.
\end{lemma}

\begin{proof}
  For $A \in \mathcal{O}$, the zigzag $A \to 0$ shows that $d_w(A,0) \leq w(A)$.
\end{proof}

  Thus, for a weight $w$ on $\mathcal{O}$, we obtain a sequence of decreasing weights on $\mathcal{O}$, $w = w_1 \geq w_2 \geq w_3 \geq \cdots$, with $w_{n+1} = \abs{d_{w_n}}_{\mathcal{O}}$.
  Say that this sequence \emph{stabilizes} if $w_{n+1} = w_n$ for some $n$.

\begin{definition} \label{def:stable-weight}
  Say the weight $w$ on $\mathcal{O}$ is \emph{stable} if $\abs{d_w}_{\mathcal{O}} = w$.
\end{definition}

\begin{lemma}
  For a metric $d$ on $\cat{A}$, $\abs{d}$ need not be an amplitude.
\end{lemma}

\begin{proof}
  Consider Example~\ref{ex:nonamplitude}, where the weight $w$ is not an amplitude and $\abs{d_w} = w$.
\end{proof}

\begin{lemma}
  For a weight $w$ on $\cat{A}$, we may have $\abs{d_w} \neq w$.
\end{lemma}

\begin{proof}
  Consider the following weight on $\cat{vect_K}$. Let $w(0) = 0$, $w(V)=1$ if $\dim(V) = 1$ and $w(V) = 3$ otherwise.
  Then $\abs{d_w}(V) = 2$ if $\dim{V} = 2$. 
\end{proof}

\subsection{Bounds on path metrics}
\label{sec:bounds-abelian-category}

  For a weight, we give an upper bound for its path metric. We define weights that give lower bounds for their path metrics. In Section~\ref{sec:equivalent-conditions} we will show that exact weights give such lower bounds.

   Let $\cat{A}$ be an additive category, together with a class, $\mathcal{O}$, of objects in $\cat{A}$, and a weight $w$ on $\mathcal{O}$ 
(see Assumption~\ref{ass:O-w}).

\begin{lemma} \label{lem:upper-bound-1}
  For all $A,B \in \cat{A}$, $d_w(A,B) \leq \abs{d_w}(A) + \abs{d_w}(B)$.
\end{lemma}

\begin{proof}
  By the triangle inequality, $d_w(A,B) \leq d_w(A,0) + d_w(0,B) = \abs{d_w}(A) + \abs{d_w}(B)$.
\end{proof}

Combining Lemmas~\ref{lem:upper-bound-1} and \ref{lem:bound-weight} we have the following.

\begin{proposition} \label{prop:upper-bound}
  For all $A,B \in \mathcal{O}$, $d_w(A,B) \leq w(A) + w(B)$.
\end{proposition}

\begin{definition} \label{def:lower-bound}
  Say that the weight $w$ \emph{lower bounds its path metric} if for all $A,B \in \mathcal{O}$, $\abs{w(A)-w(B)} \leq d_w(A,B)$.
\end{definition}

\subsection{Equivalent conditions on a weight}
\label{sec:equivalent-conditions}

We conclude this section by showing that the three conditions on a weight that we have introduced are equivalent.

\begin{theorem} \label{thm:stable-lower-bound}
  Let $\cat{A}$ be an additive category, together with a class of objects $\mathcal{O}$ in $\cat{A}$ and a weight $w$ on $\mathcal{O}$ (see Assumption~\ref{ass:O-w}). 
The weight $w$ is stable (Definition~\ref{def:stable-weight}) if and only if it lower bounds its path metric (Definition~\ref{def:lower-bound}).
\end{theorem}

\begin{proof}
  Assume that $w$ is stable.
  By the triangle inequality, for all $A,B \in \mathcal{O}$, $\abs{d_w(A,0) - d_w(B,0)} \leq d_w(A,B)$. Since $w$ is stable, we obtain $\abs{w(A) - w(B)} \leq d_w(A,B)$.

  Assume that $w$ lower bounds its path metric.
  For all $A \in \mathcal{O}$, $w(A) = \abs{w(A) - 0} = \abs{w(A) - w(0)} \leq d_w(A,0)$.
  By Lemma~\ref{lem:bound-weight}, $d_w(A,0) \leq w(A)$.
  Thus, $d_w(A,0) = w(A)$.
\end{proof}

\begin{theorem} \label{thm:equivalent-conditions}
Let $\cat{A}$ be an additive category, together with a class of objects $\mathcal{O}$ in $\cat{A}$ and a weight $w$ on $\mathcal{O}$.
Assume that for all short exact sequences $0 \to A \to B \to C \to 0$ in $\cat{A}$ in which two of $A,B,C$ are in $\mathcal{O}$ then so is the third.
The following three conditions on $w$ are equivalent:
  \begin{enumerate}
  \item \label{it:exact} $w$ is exact (Definition~\ref{def:exact-weight});
  \item \label{it:stable} $w$ is stable; and
  \item $w$ lower bounds its path metric.
  \end{enumerate}
\end{theorem}

\begin{proof}
  We will show \eqref{it:exact} iff \eqref{it:stable}. The remainder of the statement follows from Theorem~\ref{thm:stable-lower-bound}.

  First we show that \eqref{it:stable} implies \eqref{it:exact}.
  Consider a short exact sequence $0 \to A \xto{f} B \xto{g} C \to 0$ with $A,B,C \in \mathcal{O}$.
  Then by assumption and the triangle inequality $w(A) = d_w(A,0) \leq d_w(A,B) + d_w(B,0)$. By assumption $d_w(B,0) = w(B)$ and from the zigzag $A \xto{f} B$ we have that $d_w(A,B) \leq w(C)$. Thus $w(A) \leq w(C) + w(B)$.
  Similarly $w(B) = d_w(B,0) \leq d_w(B,A) + d_w(A,0) \leq w(C) + w(A)$ and
  $w(C) = d_w(C,0) \leq d_w(C,B) + d_w(B,0) \leq w(A) + w(B)$.

It remains to show that \eqref{it:exact} implies \eqref{it:stable}.
  By Lemma~\ref{lem:bound-weight}, $\abs{d_w}_{\mathcal{O}} \leq w$.
  We will obtain a contradiction to $\abs{d_w}_{\mathcal{O}} < w$.
Assume $\abs{d_w}_{\mathcal{O}} < w$.
By Definition~\ref{def:dw-object} and Lemma~\ref{lem:zigzag-zero},
there is a zigzag $\gamma$ consisting of morphisms with either zero kernel and cokernel in $\mathcal{O}$ or zero cokernel and kernel in $\mathcal{O}$ from some object $A \in \mathcal{O}$ to $0$ such that $\cost_w(\gamma) < w(A)$.
The length of any such zigzag is a nonnegative integer.
Take $\gamma$ to be such a zigzag of minimal length.
Let $f$ be the first morphism of this zigzag, which is either of the form $A \xto{f} B$ or $A \xfrom{f} B$.
Since $f$ has either zero kernel and cokernel in $\mathcal{O}$ or zero cokernel and kernel in $\mathcal{O}$,
by our assumption on $\mathcal{O}$, $B \in \mathcal{O}$.
Let $\gamma'$ denote the remainder of the zigzag $\gamma$ without the morphism $f$. Then $\gamma'$ is a zigzag from $B$ to $0$ consisting of morphisms with either zero kernel and cokernel in $\mathcal{O}$ or zero cokernel and kernel in $\mathcal{O}$.
  Since the length of $\gamma'$ is less than the length of $\gamma$, by the minimality of $\gamma$, $\cost_w(\gamma') = w(B)$. 
  There are four cases to consider, depending on the direction of $f$ and whether or not $f$ has zero kernel or zero cokernel.
  For example, if $A \xfrom{f} B$ and $f$ has zero cokernel, then we have the short exact sequence $0 \to \ker(f) \to B \to A \to 0$.
  Since $w$ is exact, $w(A) \leq w(B) + w(\ker(f)) = \cost_w(\gamma') + w(\ker(f)) = \cost_w(\gamma)$, which is a contradiction.
  In the other cases, we also have a short exact sequence containing $A$, $B$, and either $\ker f$ or $\coker f$. The same argument again gives us a contradiction.
\end{proof}

\section{Path metrics for persistence modules} \label{sec:metric-from-measure}

In this section, we specialize the results of Section~\ref{sec:weight-path-metric} to the case of persistence modules indexed by a small category whose set of objects comes equipped with a measure.

\subsection{Indexing categories with measures} \label{sec:measure}

In Section~\ref{sec:dmu}, we will show that for an indexing category $\cat{P}$ with a measure on its set of objects and a weight on an abelian category $\cat{A}$ there is an induced weight on the category of persistence modules indexed by $\cat{P}$ with values in $\cat{A}$.

Let $\cat{P}$ be a small category whose set of objects $P$ has a $\sigma$-algebra $\Omega$ and measure $\mu$. The classical case of persistence modules is given by $P \subseteq \Z$ or $P \subseteq \R$ (assumed to be measurable) with morphisms $\leq$ and the counting measure or the Lebesgue measure, respectively. The case of multi-parameter persistence modules is given by $P \subseteq \Z^d$ or $P \subseteq \R^d$ (assumed to be measurable) with the coordinate-wise/product partial order $\leq$ and the counting measure or the Lebesgue measure, respectively.

\subsection{Weights and path metrics for persistence modules}
\label{sec:dmu}

We now define an induced weight for persistence modules.
Let $\cat{P}$ be a small category whose set of objects $P$ has a measure $\mu$.
Let $\cat{A}$ be an abelian category
together with a class of objects $\mathcal{O}$ in $\cat{A}$ and a weight $w$ on $\mathcal{O}$ (see Assumption~\ref{ass:O-w}).

Assume that we have a 
persistence module $M: \cat{P} \to \cat{A}$ such that for each $p \in P$, $M(p) \in \mathcal{O}$.
Then
we have a function $w(M): P \to [0,\infty]$ given by $p \mapsto w(M(p))$. For example, if $M$ is a persistence module with values in $\cat{Vect_K}$ then $\dim(M)$ is the Hilbert function of $M$.
If $w(M)$ is $\mu$-integrable then we write $\mu(w(M))$ to denote the integral $\int_P w(M)\, d\mu$, which is also written as $\int_P w(M(p))\, d\mu(p)$.

\begin{definition} \label{def:W-mu-w}
  Consider the category of  persistence modules indexed by $\cat{P}$ with values in $\cat{A}$. 
Let $\mathcal{T}_{\mathcal{O},\mu,w}$ be the class of persistence modules $M$
    such that for all $p \in P$, $M(p) \in \mathcal{O}$ and such that $w(M)$ is $\mu$-integrable.
Then $(\mu \circ w)(M) = \mu(w(M))$ defines
    a weight $\mu \circ w$ on $\mathcal{T}_{\mathcal{O},\mu,w}$.
\end{definition}

\begin{lemma} \label{lem:muw-exact}
  If $w$ is an exact weight on $\mathcal{O}$ then $\mu \circ w$ is an exact weight on $\mathcal{T}_{\mathcal{O},\mu,w}$.
\end{lemma}

\begin{proof}
  Let $0$ be the zero persistence module. Then $(\mu \circ w)(0) = \mu(w(0)) = \int_P0\, d\mu = 0$. Also, if $M \isom N$ then for all $p \in P$, $M(p) \isom N(p)$, so $w(M(p)) = w(N(p))$, and hence $(\mu \circ w)(M) = (\mu \circ w)(N)$.

    Let $0 \to M \to N \to Q \to 0$ be a short exact sequence of persistence modules.
  Then for all $p \in P$, $0 \to M(p) \to N(p) \to Q(p) \to 0$ is a short exact sequence in $\cat{A}$.
  If $M,N,Q \in \mathcal{T}_{\mathcal{O},\mu,w}$ then for all $p \in P$, $M(p),N(p),Q(p) \in \mathcal{O}$.
  Since $w$ is an exact weight on $\mathcal{O}$, $w(M(p)) \leq w(N(p)) + w(Q(p))$.
  Thus $\int_P w(M(p)) d\mu(p) \leq \int_P w(N(p)) d\mu(p) + \int_P w(Q(p)) d\mu(p)$.
  That is, $(\mu \circ w)(M) \leq (\mu \circ w)(N) + (\mu \circ w)(Q)$.
  The other cases are similar.
\end{proof}

Similarly, we have the following.

\begin{lemma} \label{lem:muw-amplitude}
  If $\alpha$ is an amplitude on $\mathcal{O}$, then $\mu \circ \alpha$ is an amplitude on $\mathcal{T}_{\mathcal{O},\mu,w}$.
\end{lemma}

\begin{definition} \label{def:dmu}
    Combining Definitions \ref{def:dw-object} and \ref{def:W-mu-w}, we have a path metric $d_{\mu \circ w}$ on persistence modules indexed by $\cat{P}$ with values in $\cat{A}$.
\end{definition}

\begin{lemma} \label{lem:cost}
  Let $M,N$ be persistence modules indexed by $\cat{P}$ with values in $\cat{A}$ and let $\gamma$ be a zigzag in $\mathcal{T}_{\mathcal{O},\mu,w}$ from $M$ to $N$. Then
  $\cost_{\mu \circ w}(\gamma) = \mu(\cost_w(\gamma))$.
\end{lemma}

\begin{proof}
  Consider a zigzag $\gamma$ in $\mathcal{T}_{\mathcal{O},\mu,w}$ given by 
  $M \xto{f_1} M_1 \xfrom{f_2} M_2 \xto{f_3} \cdots \xfrom{f_n} N$.
  Then $\cost_{\mu \circ w}(\gamma) = \sum_{j=1}^n ((\mu \circ w)(\ker f_j) + (\mu \circ w)(\coker f_j)) = \sum_{j=1}^n (\int_P w(\ker f_j)\, d\mu + \int_P w(\coker f_j))\, d\mu = \int_P \sum_{j=1}^n (w(\ker f_j) + w(\coker f_j))\, d\mu = \int_P \cost_w(\gamma)\, d\mu$.
\end{proof}

\subsection{Bounds for the path metric on persistence modules} \label{sec:bounds}

We now provide an upper bound for the path metric induced by a weight and a lower bound on the path metric induced by an exact weight.
  Let $\cat{P}$ be a small category whose set of objects $P$ has a measure $\mu$.
Let $\cat{A}$ be an abelian category together with a class of objects $\mathcal{O}$ in $\cat{A}$ and a weight $w$ on $\mathcal{O}$ (see Assumption~\ref{ass:O-w}).

\begin{proposition} \label{prop:upper-bound-pm}
  For persistence modules $M$ and $N$ indexed by $\cat{P}$ with values in $\cat{A}$, such that for all $p \in P$, $M(p),N(p) \in \mathcal{O}$, and $w(M)$ and $w(N)$ are $\mu$-integrable, we have
  \begin{equation*}
    d_{\mu \circ w}(M,N) \leq \mu(w(M) + w(N)) = \int_P(w(M) + w(N))\, d\mu.
  \end{equation*}
\end{proposition}

\begin{proof}
  By Definition~\ref{def:W-mu-w} and Proposition~\ref{prop:upper-bound}, $d_{\mu \circ w}(M,N) \leq (\mu \circ w)(M) + (\mu \circ w)(N) = \int_P w(M)\, d\mu + \int_P w(N)\, d\mu = \int_P (w(M) + w(N))\, d\mu  = \mu(w(M) + w(N))$.
\end{proof}

\begin{theorem} \label{thm:lower-bound-pm}
Assume that $\mathcal{O}$ that satisfies the 2-of-3 property and that the weight $w$ is exact.
  For persistence modules $M,N$ indexed by $\cat{P}$ with values in $\cat{A}$, such that for all $p \in P$, $M(p),N(p) \in \mathcal{O}$, and $w(M)$ and $w(N)$ are $\mu$-integrable, we have
  \begin{equation*}
    \mu(\abs{ w(M) - w(N)}) = \int_P \abs{w(M) - w(N)}\, d\mu \leq d_{\mu \circ w}(M,N).
  \end{equation*}
\end{theorem}

\begin{proof}
  Consider a zigzag $\gamma$ in $\mathcal{T}_{\mathcal{O},\mu,w}$ given by $M = M_0 \xto{f_1} M_1 \xfrom{f_2} M_2 \xto{f_3} \cdots \xfrom{f_n} M_n = N$ such that each $f_j$ has either zero kernel or zero cokernel.
  Then for all $p \in P$, $\gamma(p)$ is a zigzag in $\mathcal{O}$ from $M(p)$ to $N(p)$.
  By Lemma~\ref{lem:cost}, $\cost_{\mu \circ w}(\gamma) = \mu(\cost_{w}(\gamma))$.
  For each $p \in P$, by Definition~\ref{def:dw-object} and Theorem~\ref{thm:equivalent-conditions},
  $\cost_{w}(\gamma(p)) \geq d_{w}(M(p),N(p)) \geq \abs{w(M(p))-w(N(p))}$.
  Therefore $\cost_{\mu \circ w}(\gamma) \geq \mu(\abs{w(M)-w(N)}) = \int_P \abs{w(M(p))-w(N(p))}\, d\mu(p)$.
  Hence, by Lemma~\ref{lem:zigzag-zero}, $d_{\mu \circ w}(M,N) \geq \mu(\abs{w(M)-w(N)})$.
\end{proof}
 
For example, for persistence modules $M$ and $N$ indexed by $(P,\mu)$ with values in $\cat{Vect_K}$ such that $\dim M$ and $\dim N$ are $\mu$-integrable, we have
\begin{equation} \label{eq:bounds}
  \int_P \abs{\dim M - \dim N}\, d\mu \leq d_{\mu \circ \dim}(M,N) \leq \int_P (\dim M + \dim N)\, d\mu.
\end{equation}

\subsection{Distance between interval modules} \label{sec:interval-modules}

In this section we compute the path distance between interval modules indexed by a totally ordered set.
Our interval modules are persistence modules indexed by $(P,\mu)$, where $P$ is a totally ordered set, and valued in $\cat{Vect_K}$.


It is a good exercise to check the following two lemmas
(or see~\cite[Appendix A]{Bubenik:2018}).

\begin{lemma} \label{lem:interval-nonzero-map}
  Let $I$ and $J$ be interval modules.
  Then there is a nonzero map $f:I \to J$ if and only if the intervals intersect and for each $a \in I$ there exists  $b \in J$ with $b \leq a$ and for each  $b \in J$ there is an $a \in I$ with $b \leq a$.
\end{lemma}

\begin{lemma} \label{lem:interval-maps}
  Let $I$ and $J$ be interval modules. Then, after possibly interchanging $I$ and $J$, we have one of the following two possible cases.
  \begin{enumerate}
    \item \label{it:a} There are maps $I \xto{f} I \cap J \xto{g} J$ with $f$ surjective, $\ker(f) = I \setminus (I \cap J)$, $g$ injective, and $\coker(g) = J \setminus (I \cap J)$. (This includes the case $I \cap J = \varnothing$.)
    \item $I \subset J$ and there is an interval module $K$ and maps $I \xfrom{f} K \xto{g} J$ with $f$ surjective, $g$ injective and $J \setminus I$ is the disjoint union of $\ker(f)$ and $\coker(g)$.
  \end{enumerate}
\end{lemma}

\begin{proposition} \label{prop:interval-module-distance}
  Let $I$, $J$ be interval modules or the zero module, which we also denote by the empty set. Then $d_{\mu \circ \dim}(I,J) = \mu(I \symmdiff J)$,
  where $I \symmdiff J$ denotes the symmetric difference $(I \cup J) \setminus (I \cap J)$.
\end{proposition} 

\begin{proof}
  $(\leq)$ If either $I$ or $J$ are zero, then we have a canonical zigzag $I \to 0$ or $0 \to J$. By Lemma~\ref{lem:interval-maps} we have one of two canonical zigzags from $I$ to $J$. In each of these cases the cost of this zigzag is $\mu(I \symmdiff J)$.

  $(\geq)$ By 
  \eqref{eq:bounds}
  $d_{\mu \circ \dim}(I,J) \geq \int \abs{\dim I - \dim J}\, d\mu = \mu(I \symmdiff J)$.
\end{proof}

\section{Wasserstein distances for Grothendieck categories} \label{sec:wasserstein}

In this section we define $p$-Wasserstein distances for a Grothendieck category and show that it generalizes the usual definition. We also show that it satisfies a universal property.

\subsection{The $p$-Wasserstein distance} \label{sec:Wpdef}

Let $\cat{A}$ be a Grothendieck category with a metric $d$ (Section~\ref{sec:metric}).
Recall (Definition~\ref{def:subcat}) that $\cat{A_{\ell}}$ is the full subcategory of objects isomorphic to direct sums of objects with local endomorphism rings.
For $1 \leq p \leq \infty$, define the \emph{$p$-Wasserstein distance} as follows.

\begin{definition} \label{def:Wp}
  Let $M, N \in \cat{A_{\ell}}$.
  Define
  \begin{equation} \label{eq:Wp}
    W_p(d)(M,N) = \inf \ \norm{ \{d(M_a,N_a)\}_{a \in A} }_p,
  \end{equation}
where the infimum is taken over all isomorphisms $M \isom \coprod_{a \in A} M_a$ and $N \isom \coprod_{a \in A} N_a$, where each $M_a$ and $N_a$ is either $0$ or has a local endomorphism ring (and is thus indecomposable).
\end{definition}

\begin{lemma} \label{lem:Wp}
  Let $M,N \in \cat{A_{\ell}}$. Assume $M \isom \coprod_{a \in A} M_a$ and $N \isom \coprod_{b \in B} N_b$, where each $M_a$ and $N_b$ has a local endomorphism ring. Then
  \begin{equation*}
    W_p(d)(M,N) = \inf_{\varphi} \norm{ \left(
        \norm{ \left( d(M_c,N_{\varphi(c)}) \right)_{c \in C} }_p,
        \norm{ \left( d(M_a,0) \right)_{a \in A-C} }_p,
        \norm{ \left( d(0,N_b) \right)_{b \in B-\varphi(C)} }_p \right) }_p,
  \end{equation*}
  where the infimum is over all matchings: $C \subset A$ and $\varphi:C \to B$ is injective.
\end{lemma}

\begin{proof}
By Theorem~\ref{thm:krsa}, the decompositions of $M$ and $N$ are unique up to isomorphism and reordering. Note that the direct sum in Definition~\ref{def:Wp} also allows zero objects. So the infimum in \eqref{eq:Wp} is over all  matchings of $A$ and $B$, where the unmatched terms are matched with the zero object. 
\end{proof}

\begin{proposition} \label{prop:Wpd-metric}
  $W_p(d)$ is a metric (Section~\ref{sec:metric}) on $\cat{A_{\ell}}$.
\end{proposition}

\begin{proof}
  By assumption, if $M \isom N$ then $d(M,N)=0$. It follows that if $M \isom N$ then $W_p(d)(M,N) = 0$.
  Since $d$ is symmetric, it follows that $W_p(d)$ is symmetric. 

The proof of the triangle inequality uses Theorem~\ref{thm:krsa}.
Let $M,N,P \in \cat{A_{\ell}}$.
Let $\eps > 0$.
By including sufficiently many zero modules and relabeling, we may assume that $M \isom \coprod_{a \in A} M_a$, $N \isom \coprod_{a \in A} N_a$, $P \isom \coprod_{a \in A} P_A$, and that
$W_p(d)(M,N) \geq \norm{\{d(M_a,N_a)\}_{a \in A}}_p - \eps$ and
$W_p(d)(N,P) \geq \norm{\{d(N_a,P_a)\}_{a \in A}}_p - \eps$.
  Then
  \begin{multline*}
    W_p(d)(M,P) \leq \norm{ \left\{ d(M_k,P_k) \right\}_k }_p \leq
    \norm{ \left\{ d(M_k,N_k) + d(N_k,P_k) \right\}_k }_p \\
    \leq
    \norm{ \left\{ d(M_k,N_k) \right\}_k }_p +
    \norm{ \left\{ d(N_k,P_k) \right\}_k }_p \leq
    W_p(d)(M,N) + W_p(d)(N,P) + 2\eps,
  \end{multline*}
  where the first inequality is by definition, the second inequality is by the triangle inequality for $d$, and the third inequality is by the Minkowski inequality. The triangle inequality follows.
\end{proof}

For example, if we have a measure space $(P,\mu)$ and a small category $\cat{P}$ with set of objects $P$, we have the Grothendieck category $\cat{Vect_K}^{\cat{P}}$ and metric $W_p(d_{\mu \circ \dim})$ on the subcategory $\cat{Vect_K}_{\ell}^{\cat{P}}$ whose objects are isomorphic to direct sums of objects with local endomorphism rings.

\subsection{The $W_p$ Isometry Theorem} \label{sec:Wp-isometry}

In this section we show that in the case of persistence modules indexed by $P \subseteq \R$ our definition of $p$-Wasserstein distance (Definition~\ref{def:Wp})
agrees with the definition using persistence diagrams (Definition~\ref{def:Wpq}). 
 Consider $\R$ with the usual total order and let $P \subseteq \R$.
  For an interval $I$ in $P$, let $P_{>I} = \{p \in P \ | \ \forall x \in I, x < p\}$.
  Let $\mu$ be a measure on $P$ such that for all intervals $I$ in $P$, $\mu(I) = \inf P_{>I} - \inf I$, where $\inf \emptyset = \infty$.
  For example, we may take $P = \R$ or $P = [0,\infty)$ with the Lebesgue measure, or $P = \Z$ or $P = \N$ with the counting measure.

Recall (Section~\ref{sec:wasserstein-pers-mod}) that for an interval module $I$, $x(I) = (\inf I, \inf P_{>I})$ and that $\Delta$ denotes the diagonal in $[-\infty,\infty]^2$.
Also, for $x,y \in [-\infty,\infty]^2$, $d(x,y) = \norm{x-y}_1$.

\begin{lemma} \label{lem:lambda-I}
  Let $I$ be an interval in $P$. Then $d(x(I),\Delta) = \mu(I)$.
\end{lemma}

\begin{proof}
$d(x(I),\Delta) = d((\inf I, \inf P_{>I}),\Delta) = \inf P_{>I} - \inf I = \mu(I)$.
\end{proof}

\begin{lemma} \label{lem:intersect}
  If $I,J$ are 
  intervals in $P$ with $I \cap J \neq \varnothing$, then
  $d(x(I),x(J)) = \mu(I \symmdiff J)$,
  where $I \symmdiff J$ denotes the symmetric difference $(I \cup J) \setminus (I \cap J)$.
\end{lemma}

\begin{proof}
  There are a number of cases to consider.
  However, in each case,
  $\mu({I \symmdiff J}) = \abs{\inf I - \inf J} + \abs{\inf P_{>I} - \inf P_{>J}} = \norm{x(I)-x(J)}_1 = d(x(I),x(J))$.
\end{proof}

\begin{lemma} \label{lem:disjoint}
  If $I$ and $J$ are 
  intervals in $P$ with $I \cap J = \varnothing$, then $d(x(I),x(J)) \geq \mu(I) + \mu(J)$.
\end{lemma}

\begin{proof}
  Without loss of generality, assume that $\inf I \leq \inf P_{>I} \leq \inf J \leq \inf P_{>J}$. Then
  $d(x(I),x(J)) = \inf P_{>J} - \inf P_{>I} + \inf J - \inf I \geq \inf P_{>J} - \inf J + \inf P_{>I} - \inf I = \mu(I) + \mu(J)$.
\end{proof}
 
\begin{proposition} \label{prop:W11IJ}
  For intervals $I$ and $J$ in $P$, $W_1(I,J) = \mu(I \symmdiff J)$.
\end{proposition}

\begin{proof}
  There are only two matchings between $I$ and $J$:  one in which $I$ and $J$ are matched to one another, and one in which $I$ and $J$ are both matched to the diagonal.
  So by Definition~\ref{def:Wpq} and Lemma~\ref{lem:lambda-I},
  \begin{eqnarray*}
  	W_1(I,J) & = & \min \left( d(x(I),x(J)), d(x(I),\Delta) + d(\Delta,x(J)) \right)
  		\\
  		& = & \min \left( d(x(I),x(J)), \mu(I) + \mu(J) \right).
  \end{eqnarray*}
  If $I \cap J \neq \varnothing$, then by Lemma~\ref{lem:intersect}, $d(x(I),x(J)) = \mu(I \symmdiff J) \leq \mu(I) + \mu(J)$, so $W_{1}(I,J) = \mu(I \symmdiff J)$.
  If $I \cap J = \varnothing$, then by Lemma~\ref{lem:disjoint} it follows that $W_{1}(I,J) = \mu(I) + \mu(J) = \mu(I \symmdiff J)$.
\end{proof}

\begin{theorem}[$W_p$ Isometry Theorem] \label{thm:Wp-isometry}
  Let $P \subseteq \R$ with measure $\mu$ such that for each interval $I$ in $P$, $\mu(I) = \inf P_{>I} - \inf I$.
  If $M,N \in \cat{Vect_K}^{\cat{P}}$ have a persistence diagram, then for $1 \leq p \leq \infty$,
  \begin{equation*}
    W_p(d_{\mu \circ \dim})(M,N)) = \inf \norm{ \{\mu(M_a \symmdiff N_a)\}_{a \in A}}_p = W_p(M,N),
  \end{equation*}
  where the infimum is taken over all isomorphisms $M \isom \coprod_{a \in A} M_a$ and $N \isom \coprod_{a \in A} N_{a}$ where every $M_a$ and $N_a$ is either an interval module or is zero, which corresponds to the empty set.
\end{theorem}

\begin{proof}
  The first equality follows from Definition~\ref{def:Wp} and Proposition~\ref{prop:interval-module-distance}.

Assume $M \isom \coprod_{a \in A} I_a$ and $N \isom \coprod_{b \in B} I'_b$, where each $I_a$ and $I'_b$ is an interval module. 
By Definition~\ref{def:Wpq} and Lemma~\ref{lem:lambda-I},
  \begin{multline*}
    W_p(M,N) =\\ \inf_{\varphi} \norm{ \left(
        \norm{ \left\{ d(x(I_c),x(I'_{\varphi(c)})) \right\}_{c \in C} }_p,
        \norm{ \left\{ \mu(I_a) \right\}_{i \in A-C} }_p,
        \norm{ \left\{ \mu(I'_b) \right\}_{j \in B-\varphi(C)} }_p \right) }_p,
  \end{multline*}
where the infimum is over all matchings $\varphi$ between $A$ and $B$.
By Lemma~\ref{lem:disjoint}, this equals the infimum taken over matchings $\varphi:C \to B$ with the property that $I_c \cap I'_{\varphi(c)} \neq \varnothing$ for all $c \in C$ (where it could be that $C = \varnothing$).
  Thus, by Lemma~\ref{lem:intersect},
  \begin{multline*}
    W_p(M,N) =\\ \inf_{\varphi} \norm{ \left(
        \norm{ \left\{ \mu(I_c \symmdiff I'_{\varphi(c)}) \right\}_{c \in C} }_p,
        \norm{ \left\{ \mu(I_a \symmdiff \varnothing) \right\}_{i \in A-C} }_p,
        \norm{ \left\{ \mu(\varnothing \symmdiff I'_b) \right\}_{j \in B-\varphi(C)} }_p \right) }_p.
  \end{multline*}
  Writing this more compactly we obtain the second equality.
\end{proof}

\subsection{The universal property of  $W_p(d)$} \label{sec:universal-property}

In this section we show that $W_p(d)$ may be characterized as the largest $p$-subadditive metric that is is bounded by $d$ on those objects with local endomorphism rings.
Let $\cat{A}$ be a Grothendieck category with metric $d$ (Section~\ref{sec:metric}).
Let $1 \leq p \leq \infty$.

\begin{definition} \label{def:dp}
  For $A,B \in \cat{A}$, let $d_p(A,B) = \min(d(A,B),\norm{(d(A,0),d(0,B))}_p)$.
\end{definition}

One may check that $d_p$ is a metric on $\cat{A}$ (see~\cite[Lemma 3.13]{be:universality}).

\begin{lemma} \label{lem:restricted}
  Restricted to objects with local endomorphism rings and zero, $W_p(d)$ equals $d_{p}$.
\end{lemma}

\begin{proof}
  Consider $M,N$ with local endomorphism rings or being zero.
  By Definitions \ref{def:Wp} and \ref{def:dp},
  $W_p(d)(M,N) = \min\left( d(M,N), \norm{(d(M,0), d(0,N) )}_p \right) = d_{p}(M,N)$.
\end{proof}

\begin{definition} \label{def:dGammaack}
  Say that a metric $d$
  on $\cat{A_{\ell}}$
  is \emph{$p$-subadditive}
  if
  for any sets $\{M_a\}_{a \in A}$ and $\{N_a\}_{a \in A}$ of objects in $\cat{A_{\ell}}$,
$d(\coprod_{a \in A}M_a,\coprod_{a \in A}N_a) \leq \norm{\{d(M_a,N_a)\}_{a \in A}}_p.$
\end{definition}

\begin{proposition}\label{prop:was-sub}
  $W_p(d)$ is a $p$-subadditive metric on $\cat{A_{\ell}}$.
\end{proposition}

\begin{proof}
  Consider $\coprod_{a \in A}M_a$ and $\coprod_{a \in A}N_a$, where $M_a,N_a \in \cat{A_{\ell}}$ for all $a \in A$.
  For the left hand side, $W_p(d)(\coprod_{a \in A}M_a, \coprod_{a \in A}N_a) = \inf \norm{\{d(P_s,Q_s)\}_{s \in S}}_p$, where $\coprod_{a \in A}M_a \isom \coprod_{s \in S}P_s$ and $\coprod_{a \in A}N_a \isom \coprod_{s \in S}Q_s$ with each $P_s$ and $Q_s$ either having a local endomorphism ring or being zero.
  For the right hand side, $\norm{\{W_p(d)(M_a,N_a)\}_{a \in A}}_p = \inf \norm{\{d(P_{a,s},Q_{a,s})\}_{a \in A,s \in B_a}}_p$, where $M_a \isom \coprod_{s \in B_a}P_{a,s}$ and $N_a \isom \coprod_{s \in B_a}Q_{a,s}$ with each $P_{a,s}$ and $Q_{a,s}$ either having a local endomorphism ring or being zero.
  By Theorem~\ref{thm:krsa} each term in the right hand side is a term in the left hand side.
  The result follows.
\end{proof}

\begin{proposition}\label{prop:sub}
  Let $d'$ be a $p$-subadditive metric on $\cat{A_{\ell}}$ that
  is bounded above by $d$
  on objects with local endomorphism rings and zero. Then $d' \leq W_p(d)$.
\end{proposition}

\begin{proof}
  Let $M,N \in \cat{A_{\ell}}$. Consider Definition~\ref{def:Wp}.
  For each pair of isomorphisms $M \isom \coprod_{a \in A}M_a$ and $N \isom \coprod_{a \in A}N_a$ where each $M_a$ or $N_a$ is either $0$ or has a local endomorphism ring, since $d'$ is $p$-subadditive,
$d'(M,N)
\leq \norm{\{d'(M_a,N_a)\}_{a \in A}}_p$,
  which by assumption is bounded above by
  $\norm{\{d(M_a,N_a)\}_{a \in A}}_p$.
Therefore $d'(M,N) \leq W_p(d)(M,N)$.
\end{proof}

Combining Lemma~\ref{lem:restricted} and Propositions \ref{prop:was-sub} and \ref{prop:sub}, we have the following.

\begin{theorem}[Universal characterization of $W_p(d)$] \label{thm:universal}
  $W_p(d)$ is the largest $p$-subadditive metric on $\cat{A_{\ell}}$ that
  is bounded above by $d$ on objects with local endomorphism rings and zero.
\end{theorem}

\section{Algebra and persistence modules} \label{sec:agreement}


In this section we will prove that $W_1(d_{\mu \circ \dim})$ and $d_{\mu \circ \dim}$ are equal for certain persistence modules.
Along the way, we will prove structure theorems for maps from an interval module and maps to an interval module and show that both
monomorphisms and epimorphisms of persistence modules
induce algebraic matchings of direct summands.
 Let $P \subseteq \R$.
  Let $\mu$ be a measure on $P$ such that for all intervals $I$ in $P$, $\mu(I) = \inf P_{>I} - \inf I$, where
$P_{>I} = \{p \in P \ | \ \forall x \in I, x < p\}$ and
$\inf \emptyset = \infty$.

Throughout this section (with the exception of Definition~\ref{def:coherent-bases}), we will restrict $\cat{Vect_K}^{\cat{P}}$ to the full subcategory, $\cat{Vect^P_{ds}}$, whose objects are isomorphic to direct sums of interval modules.
  Recall that $\mu \circ \dim$ is a weight on the persistence modules whose Hilbert functions are integrable.
  It restricts to a weight on $\cat{Vect^P_{ds}}$.
  We obtain a corresponding path metric $d_{\mu \circ \dim}$ on $\cat{Vect^P_{ds}}$.

\subsection{Change of bases} \label{sec:base-change}
In this section we give a change-of-basis lemma that is a main technical ingredient in our proof of induced algebraic matchings and hence of our $W_1$ isometry theorem. To help with the arguments used in that proof, we give two  examples that use this lemma.

\begin{definition} \label{def:coherent-bases}
  Consider $M \in \cat{Vect_K}^{\cat{P}}$.
  For each $a \in P$, let $B_a$ be a basis for $M(a)$.
  Call $\{B_a\}_{a \in P}$ a \emph{set of coherent bases} for $M$ if for all $a \leq b \in P$, $M(a \leq b)$ restricts to a matching of $B_a$ and $B_b$. That is, there is a subset $S \subseteq B_a$ such that $M(a \leq b)|_S$ is one-to-one and has its image in $B_b$ and $M(a \leq b)|_{B_a \setminus S} = 0$.
\end{definition}


We remark that a set of coherent bases for a persistence module is often visualized as a set of intervals called a \emph{barcode}.

\begin{notation} \label{not:leq}
Following~\cite[Definition 9]{Bubenik:2018}, for intervals $I,J \subseteq P$ or corresponding interval modules say that $I \leq J$ if for all $i \in I$ there exists $j \in J$ such that $i \leq j$ and if for all $j \in J$ there exists $i \in I$ such that $i \leq j$.
\end{notation}

\begin{lemma}[Change of basis lemma] \label{lem:base-change}
  Let $M = I \oplus J$, where $I,J$ are interval modules, $I \leq J$ and $I \cap J \neq \emptyset$.
  Let $\{\{e_c\}\}_{c \in I}$ and $\{\{f_c\}\}_{c \in J}$ denote sets of coherent bases for $I$ and $J$, respectively.
  Consider $ke_c + \ell f_c$, where $c \in I \cap J$ and $k,\ell \in K \setminus \{0\}$.
  Then $M$ has a set of coherent bases given by $\{\{k e_c\}\}_{c \in I \setminus J} \cup \{\{k e_c, ke_c + \ell f_c\}\}_{c \in I \cap J} \cup \{\{\ell f_c\}\}_{c \in J \setminus I}$.
\end{lemma}

\begin{proof}
  It remains to show that the maps $M(c \leq d): M(c) \to M(d)$ restrict to a matching of bases.
  If $I \setminus J \neq \emptyset$ then let $x \in I \setminus J$,
  let $y \in I \cap J$, and if $J \setminus I \neq \emptyset$ then let $z \in J \setminus I$.
  Then $M(x \leq y)(k e_x) = k e_{y}$,
  $M(y \leq z)(k e_{y}) = 0$, and
  $M(y \leq z)(ke_{y} + \ell f_{y}) = \ell f_{z}$.
\end{proof}

\begin{example} \label{ex:one-to-two}
  Consider $f:M \to N$, where $N = N_1 \oplus N_2$, $M, N_1, N_2$ are interval modules, $N_1 \leq N_2 \leq M$, and $M \cap N_1 \neq \emptyset$. Let 
$\{e_c\}_{c \in M}$,
$\{e'_c\}_{c \in N_1}$, and
$\{e''_c\}_{c \in N_2}$ be coherent sets of bases for $M, N_1, N_2$.
Assume that $f(e_c) = ke'_c + \ell e''_c$ for some $c \in M \cap N_1$ where $k,\ell \neq 0$.
It follows that
$f(e_c) = ke'_c + \ell e''_c$ for all $c \in M \cap N_1$ 
and that
$f(e_c) = \ell e''_c$ for all $c \in N_2 \setminus N_1$.

Apply Lemma~\ref{lem:base-change} to write $N$ as the internal direct sum $N_1 \oplus N'_2$, where $N'_2$ has a set of coherent bases given by $\{ke'_c + \ell e''_c\}_{c \in N_1 \cap N_2} \cup \{\ell e''_c\}_{c \in N_2 \setminus N_1}$.
Let $p_1$, $p'_2$ denote the canonical maps to the direct summands in $N_1 \oplus N'_2$ and let
$i_1$, $i'_2$ denote the canonical maps from the direct summands to $N_1 \oplus N'_2$.
Then $f = i'_2p'_2f$.
Since $i_1p_1 + i'_2p'_2 = 1_N$ and $i_1$ is a monomorphism it follows that $p_1f = 0$.
\end{example}

\begin{example} \label{ex:two-to-one}
  Consider $f:M \to N$ where $M= M_1 \oplus M_2$, $M_1, M_2, N$ are interval modules $N \leq M_1 \leq M_2$ and $N \cap M_2 \neq \emptyset$.
Let $\{e_c\}_{c \in N}$, $\{e'_c\}_{c \in M_1}$ and $\{e''_c\}_{c \in M_2}$ be sets of coherent bases for $N, M_1, M_2$.
Assume that $f(e'_c) = ke_c$ for all $c \in M_1 \cap N$, where $k \neq 0$ and 
$f(e''_c) = \ell e_c$ for all $c \in M_2 \cap N$, where $\ell \neq 0$.

Apply Lemma~\ref{lem:base-change} to write $M$ as the internal direct sum $M_1 \oplus M'_2$, where $M'_2$ has a set of coherent bases given by $\{e''_c - \ell k^{-1} e'_c\}_{c \in M_1 \cap M_2} \cup \{e''_c\}_{c \in M_2 \setminus M_1}$.
Then $fi'_2 = 0$, where $i'_2:M'_2 \to M_1 \oplus M'_2$ is the canonical map.
\end{example}

\subsection{Structure theorems} \label{sec:structure}

In this section we give structure theorems for maps out of and into an interval module.

\begin{notation} \label{not:Supset}
Given two intervals $I$ and $J$, write $I \Subset J$ if $I \subset J$ and there exist $a,b \in J$ such that for all $i \in I$, $a < i < b$.
We will also denote this by $J \Supset I$.
\end{notation}

Given a persistence module,
$M = N \oplus \bigoplus_{j=1}^{\infty} M_j$, or 
$M = N \oplus \bigoplus_{j=1}^n M_j$,
let $i_N:N \to M$, $p_N: M \to N$ denote the canonical maps.
Similarly, for all $j$, let $i_j:M_j \to M$ and $p_j:M \to M_j$ denote the canonical maps.
Recall Notations \ref{not:leq} and \ref{not:Supset}.

\begin{theorem}[Structure theorem for maps from an interval module] \label{thm:coker}
  Let $M$ be a direct sum of interval modules (with arbitrary indexing set) and let $I$ be an interval module.
  Given a nonzero map $f\colon I \rightarrow M$, there exists
  an isomorphism 
    $\theta: M \to N \oplus N'$ with
    \begin{enumerate}
    \item $N' = \bigoplus_{j=1}^n M_j$ for some $n \geq 1$, or
    \item $N' = \bigoplus_{j=1}^{\infty} M_j$,
    \end{enumerate}
  \noindent
  such that
  $p_N\theta f = 0$ and
  for all $j$,
  $M_j$ is an interval module with $M_j \leq I$,
  $M_j \cap I \neq \emptyset$,
  $p_j \theta f$ is nonzero,
  and the interval $M_j \cap I$ contains the interval $M_{j+1} \cap I$.

  In the first case,
    $M_1 \Supset M_2 \Supset \cdots \Supset M_n$
    and if $I$ does not have a lower bound then $n=1$.
If $\inf I \in I$ then only the first case can occur.
For the second case, $\lim_{n \to \infty} (\sup M_j) = \inf I$.

  In both cases, $\ker f = I \setminus M_1$.
  In the first case,
\begin{equation*}
  \coker f = N \oplus (M_n \setminus I) \oplus \bigoplus_{j=1}^{n-1} M_j \setminus ((M_j \setminus M_{j+1}) \cap I).
\end{equation*}
\end{theorem}

\begin{proof}
  Assume $M = \coprod_{\alpha \in A} M_{\alpha}$ where $M_{\alpha}$ is an interval module.
  If $p_{\alpha}f$ is nonzero for some $\alpha \in A$
then $M_{\alpha} \cap I \neq \emptyset$ and $M_{\alpha} \leq I$.
Furthermore, there is a set of coherent bases $\{\{e_c\}\}_{c \in I}$ for $I$ and a set of coherent bases $\{\{f_d\}\}_{d \in M_{\alpha}}$ for $M_{\alpha}$.

For each $c \in I$, let $A_c = \{ \alpha \in A \ | \ p_{\alpha}f(e_c) \neq 0\}$.
  By the definition of direct sum, $\abs{A_c} < \infty$.
  If $c \leq d$ then $A_c \supseteq A_d$.
  Let $A' = \bigcup_{c \in I} A_c$.
  Since $A'$ is a directed union of finite sets, $A'$ is countable.

  Since for all $\alpha \in A'$, $M_{\alpha} \leq I$, for each $\alpha,\beta \in A'$, $M_{\alpha} \cap M_{\beta} \neq \emptyset$. Order $A'$ by the right ends of the intervals. That is, $\{M_{\alpha}\}_{\alpha \in A'} = \{M_j\}_{j=1}^{\infty}$ or $\{M_{\alpha}\}_{\alpha \in A'} = \{M_j\}_{j=1}^n$ such that
  for all $j$, $M_j \cap I \supset M_{j+1} \cap I$.

  For all $j$ either $M_{j+1} \leq M_j$ or $M_{j+1} \Subset M_j$.
  In the case that $\{M_{\alpha}\}_{\alpha \in A'} = \{M_j\}_{j=1}^n$,
  whenever $M_{j+1} \leq M_j$,
  we can apply Lemma~\ref{lem:base-change} as in Example~\ref{ex:one-to-two}
  so that we may remove $M_{j+1}$ from our list.
  By induction, we have 
  $M_1 \Supset M_2 \Supset \cdots \Supset M_{n'}$.

If $I$ does not have a lower bound then $M_i \leq I$, $M_j \leq I$ and $M_i \Subset M_j$ leads to a contradiction. 

For each $c \in I$, by the definition of direct sum, $p_j \theta f(c) \neq 0$ for only finitely many $j$. It follows that if $\inf I \in I$ then one has the case of only finitely many $M_j$ and that if one has infinitely many $M_j$ then
$\lim_{n \to \infty}(\sup M_j) = \inf I$.

Finally, $I$ has a set of coherent bases $\{\{e_c\}\}_{c \in I}$ 
and each $M_j$ has a set of coherent bases $\{\{e_{j,c}\}\}_{c \in M_j}$
such that 
for $c \in (M_j \cap I) \setminus (M_{j+1} \cap I)$,
$\theta f(e_c) = e_{1,c} + \cdots + e_{j,c}$.
It follows that $\ker f$ and $\coker f$ are as claimed.
\end{proof}

\begin{corollary} \label{cor:coker}
  Given a short exact sequence $0 \to I \to M \to N \to 0$ with $I$ an interval module and $M$ a finite direct sum of interval modules, it follows that 
  $$W_1(d_{\mu \circ \dim})(M,N) \leq \mu(I).$$
\end{corollary}

\begin{proof}
  Let $f$ denote the given map $I \to M$.
  Apply Theorem~\ref{thm:coker} with $\ker f = 0$.
  We have
  $M \isom N' \oplus \bigoplus_{j=1}^{n} M_j$,
where each $M_j$ is an interval module with $M_j \leq I$, 
and for all $j$, $M_j \cap I \neq \emptyset$ and $M_j \cap I \supset M_{j+1} \cap I$.
Furthermore,
\[
  N \isom N' \oplus (M_n \setminus I) \oplus \bigoplus_{j=1}^{n-1} M_j \setminus ((M_j \setminus M_{j+1}) \cap I).
  \]
  It follows that 
\[    W_1(d_{\mu \circ \dim})(M,N) \leq \mu(M_n \cap I) + \sum_{j=1}^{n-1} \mu((M_j \setminus M_{j+1})\cap I) = \mu(M_1 \cap I) = \mu(I).\qedhere
\]
\end{proof}

In the dual case we have the following.

\begin{theorem}[Structure theorem for maps to an interval module] \label{thm:ker}
  Let $M$ be a direct sum of interval modules and let $I$ be an interval module.
  Given a nonzero map $f: M \to I$, there exists
  an isomorphism
    $\theta: M \to N \oplus \bigoplus_{\alpha \in A} M_{\alpha}$ such that
    $f \theta i_N = 0$ and for all $\alpha \in A$, $I \leq M_{\alpha}$, $M_{\alpha} \cap I \neq \emptyset$, and $f \theta i_{\alpha}$ is nonzero.
    It follows that $\coker f = I \setminus \bigcup_{\alpha \in A} M_{\alpha}$.

If $A$ is finite then $\bigoplus_{\alpha \in A} M_{\alpha} \isom \bigoplus_{j=1}^n M_j$ for some $n \geq 1$, where
$M_1 \Supset M_2 \Supset \cdots \Supset M_n$,
and 
if $I$ does not have an upper bound then $n=1$.
Furthermore
$\coker f = I \setminus M_1$ and
\begin{equation*}
  \ker f = N \oplus (M_n \setminus I) \oplus \bigoplus_{j=1}^{n-1} M_j \setminus ((M_j \setminus M_{j+1}) \cap I).
\end{equation*}
\end{theorem}

\begin{proof}
  Assume $M = \coprod_{\alpha \in B} M_{\alpha}$ where $M_{\alpha}$ is an interval module.
  Let $A  = \{\alpha \in B \ | \ f i_{\alpha} \text{ is nonzero}\}$.
  Let $N = \coprod_{\alpha \in B \setminus A} M_{\alpha}$.
  For all $\alpha \in A$,
  $M_{\alpha} \cap I \neq \emptyset$
  and
  $I \leq M_{\alpha}$.

  Now assume that $A$ is finite.
    Order the elements of $\{M_{\alpha}\}_{\alpha \in A}$
    by their left ends. That is,
  for some $n \geq 1$, we have
  $\{M_j\}_{j=1}^n$ where
  $M_1 \cap I \supset \cdots \supset M_n \cap I$.
For all $j$ either $M_j \leq M_{j+1}$ or $M_j \Supset M_{j+1}$.
Whenever $M_j \leq M_{j+1}$
apply Lemma~\ref{lem:base-change} as in Example~\ref{ex:two-to-one} so that we may remove $M_{j+1}$ from our list.
By induction, we 
obtain 
$M_1 \Supset M_2 \Supset \cdots \Supset M_{n'}$.
If $I$ does not have an upper bound then $I \leq M_i$, $I \leq M_j$ and $M_i \Subset M_j$ leads to a contradiction. 
\end{proof}

\begin{corollary} \label{cor:ker}
  Given a short exact sequence $0 \to M \to N \to I \to 0$, where $I$ is an interval module and $N$ is a
  finite
  direct sum of interval modules, it follows that $W_1(d_{\mu \circ \dim})(M,N) \leq \mu(I)$.
\end{corollary}

\begin{proof}
  Let $f$ denote the given map $N \to I$.
  Apply Theorem~\ref{thm:ker} with $\coker f = 0$.
  We have
  $N \isom N' \oplus \bigoplus_{j=1}^{n} M_j$, 
where each $M_j$ is an interval module with $I \leq M_j$, and
$M_1 \Supset M_2 \Supset \cdots \Supset M_n$. 
Furthermore,
$$
M \isom N' \oplus (M_n \setminus I) \oplus \bigoplus_{j=1}^{n-1} M_j \setminus ((M_j \setminus M_{j+1}) \cap I).
$$
  It follows that 
\[
W_1(d_{\mu \circ \dim})(M,N) \leq \mu(M_n \cap I) + \sum_{j=1}^{n-1} \mu((M_j \setminus M_{j+1})\cap I) = \mu(M_1 \cap I) = \mu(I). \qedhere
\]
\end{proof}

\subsection{Induced algebraic matching} \label{sec:matching}

In this section we show that for monomorphisms and epimorphisms of persistence modules there is an induced algebraic matching of interval modules.

  Say that two intervals $I$ and $J$ have the \emph{same right end} if $\sup I = \sup J$ and $\sup I \in I$ iff $\sup J \in J$.

\begin{theorem}[Induced algebraic matching for monomorphisms] \label{thm:matching-mono}
  Let $f:M \to N$ be a monomorphism between persistence modules with direct-sum decompositions into finitely many interval modules.
  Then there are internal direct sum decompositions $M = \bigoplus_{a \in A} M_a$ and $N = \bigoplus_{a \in A} N_a$ where
  each $M_a$ is either an interval module or zero and
  each $N_a$ is an interval module such that following hold.
For all $a \in A$, if $M_a$ is nonzero then $M_a$ and $N_a$ have the same right end,
  $p'_afi_a: M_a \to N_a$ is a monomorphism, where $i_a:M_a \to M$ and $p'_a:N \to N_a$ are the canonical maps,
for all other interval modules $N_b$ with the same right end as $M_a$, $p'_b fi_a = 0$
and for all other interval modules $M_b$ with the same right end as $N_a$, $p'_afi_b = 0$.
\end{theorem}

\begin{proof}
  Let $M = \bigoplus_{k=1}^m M_k$ and $N = \bigoplus_{j=1}^n N_j$. 
  The map $f$ determines and is determined by the maps $f_{j,k} := p'_j f i_k$, where $i_k: M_k \to M$ and $p'_j: N \to N_j$ are the canonical maps.
  Our proof is by a matrix reduction argument.
  Since $f$ is a monomorphism, for each $M_k$ there exists an $N_j$ with the same right end such that $M_k \subseteq N_j$ and $f_{j,k}$ is nonzero
(see Lemma~\ref{lem:interval-nonzero-map} and  Lemma~\ref{lem:interval-maps}\eqref{it:a}).

  Partition the intervals in $\{M_k\}_{k=1}^m$ and $\{N_j\}_{j=1}^n$ into subsets with the same right end. 
Use this partition to order the $\{M_k\}$ and $\{N_j\}$.
For the $\{M_k\}$ and $\{N_j\}$ with the same right end, order them by reverse-inclusion and inclusion, respectively. 

  Consider one of the blocks $\{M_k\}$, $\{N_j\}$ with the same right end.
  Choose $k_1$ so that $M_{k_1}$ is a largest interval. 
Let $N_{j_1}$ be a smallest one in the block with $f_{j,k_1}$ nonzero.
Apply Lemma~\ref{lem:base-change} iteratively to $N_{j_1}$ and the other $N_j$ in the block for which $f_{j,k_1}$ is nonzero (see Example~\ref{ex:one-to-two}). We obtain a basis for $N$ such that $f_{j_1,k_1}$ is nonzero and $f_{j,k_1}$ is zero for the other $N_j$ in the block.
Reorder the $N_j$ in the block so that $N_{j_1}$ is first. 
Next, apply Lemma~\ref{lem:base-change} iteratively to $M_{k_1}$ and the other $M_k$ in the block for which $f_{j_1,k}$ is nonzero (see Example~\ref{ex:two-to-one}). We obtain a basis for $M$ such that $f_{j_1,k_1}$ is nonzero and $f_{j_1,k}$ is zero for the other $M_k$ in the block.

Now consider a next largest $M_{k_2}$ in the block.
Since $f$ is a monomorphism, there is a smallest $N_{j_2}$ with $j_2 \neq j_1$ such that $f_{j_2,k_2}$ is nonzero. 
Again apply Lemma~\ref{lem:base-change} iteratively to obtain a basis for $N$ such that $f_{j_2,k_2}$ is nonzero and $f_{j,k_2}$ is zero for the $N_j$ in the block with $j \neq j_2$.
Reorder the $N_j$ in the block so that $N_{j_2}$ is second.
Also, apply Lemma~\ref{lem:base-change} iteratively to obtain a basis for $M$ such that $f_{j_2,k_2}$ is nonzero and $f_{j_2,k}$ is zero for the $M_k$ in the block with $k \neq k_2$.
Continue in the same way for the remainder of the $M_k$ in the block.
Repeat for each of the blocks.

For each $M_k$, let $N_k$ be the corresponding direct summand of $N$ obtained by the above procedure.
For the remaining $N_j$, let $M_j=0$.
\end{proof}

\begin{corollary}
  \label{cor:mono}
  Let $f:M \to N$ be a monomorphism between persistence modules with direct-sum decompositions into finitely many interval modules.
  Then $W_1(d_{\mu \circ \dim})(M,N) \leq \int_P \dim(\coker f) \, d\mu$.
\end{corollary}

\begin{proof}
  By Theorem~\ref{thm:matching-mono}, $M = \bigoplus_a M_a$ and $N = \bigoplus_a N_a$ where each $M_a$ is an interval module or zero and each $N_a$ is an interval module, and $f_a := p'_a f i_a$ is a monomorphism.
  Note that $M_a$ and $N_a$ have the same right ends and that $\coker f_a = N_a \setminus M_a$.
  We remark that there may be $b \neq a$ such that $p'_b f i_a$ is nonzero (see Theorem~\ref{thm:coker}).

  By the rank-nullity theorem, $\int_P \dim (\coker f) \, d\mu =
  \int_P (\dim N - \dim M) \, d\mu = \sum_a \int_P (\dim N_a - \dim M_a) \, d\mu = \sum_a \int_P \dim (N_a \setminus M_a) \, d\mu = \sum_a \int_P \dim (\coker f_a) \, d\mu = \sum_a d_{\mu \circ \dim}(M_a,N_a)$.
Therefore $W_1(d_{\mu \circ \dim})(M,N) \leq \int_P \dim (\coker f) \, d\mu$.
\end{proof}

The following is the Matlis dual~\cite[Section 2.5]{Miller:2020c} of Theorem~\ref{thm:matching-mono}, and the result follows by Matlis duality. However, we give an independent, elementary proof.
Say that two intervals $I$ and $J$ have the \emph{same left end} if $\inf I = \inf J$ and $\inf I \in I$ iff $\inf J \in J$.

\begin{theorem}[Induced algebraic matching for epimorphisms]  \label{thm:matching-epi}
  Let $f:M \to N$ be an epimorphism between persistence modules with direct-sum decompositions into finitely many interval modules.
  Then there are internal direct sum decompositions $M = \bigoplus_{a \in A} M_a$ and $N = \bigoplus_{a \in A} N_a$ where
  each $M_a$ is an interval module and
  each $N_a$ is either an interval module or zero such that the following hold.
  For all $a \in A$,
  if $N_a$ is nonzero then $M_a$ and $N_a$ have the same left end,
  $p'_afi_a: M_a \to N_a$ is an epimorphism, where $i_a:M_a \to M$ and $p'_a:N \to N_a$ are the canonical maps,
  for all other interval modules $M_b$ with the same left end as $N_a$, $p'_a f i_b = 0$ and for all other interval modules $N_b$ with the same left end as $M_a$, $p'_b f i_a = 0$.
\end{theorem}

\begin{proof}
  Let $M = \bigoplus_{k=1}^m M_k$ and $N = \bigoplus_{j=1}^n N_j$. 
  The map $f$ determines and is determined by the maps $f_{j,k} := p'_j f i_k$,
  where $i_k: M_k \to M$ and $p'_j: N \to N_j$ are the canonical maps.
  Our proof is by a matrix reduction argument.
  Since $f$ is an epimorphism, for each $N_j$ there exists an $M_k$ with the same left end such that $N_j \subseteq M_k$ and $f_{j,k}$ is nonzero
(see Lemma~\ref{lem:interval-nonzero-map} and  Lemma~\ref{lem:interval-maps}\eqref{it:a}).

  Partition the intervals in $\{M_k\}_{k=1}^m$ and $\{N_j\}_{j=1}^n$ into subsets with the same left end. 
Use this partition to order the $\{M_k\}$ and $\{N_j\}$.
For the $\{M_k\}$ and $\{N_j\}$ with the same left end, order them by inclusion and reverse-inclusion, respectively. 

  Consider one of the blocks $\{M_k\}$, $\{N_j\}$ with the same left end.
  Choose $j_1$ so that $N_{j_1}$ is a largest interval. 
Let $M_{k_1}$ be a smallest interval in the block with $f_{j_1,k}$ nonzero.
Apply Lemma~\ref{lem:base-change} iteratively to $M_{k_1}$ and the other $M_k$ in the block for which $f_{j_1,k}$ is nonzero (see Example~\ref{ex:two-to-one}). We obtain a basis for $M$ such that $f_{j_1,k_1}$ is nonzero and $f_{j_1,k}$ is zero for the other $M_k$ in the block.
Reorder the $M_k$ in the block so that $M_{k_1}$ is first. 
Next, apply Lemma~\ref{lem:base-change} iteratively to $N_{j_1}$ and the other $N_j$ in the block for which $f_{j,k_1}$ is nonzero (see Example~\ref{ex:one-to-two}). We obtain a basis for $N$ such that $f_{j_1,k_1}$ is nonzero and $f_{j,k_1}$ is zero for the other $N_j$ in the block.

Now consider a next largest $N_{j_2}$ in the block.
Since $f$ is an epimorphism, there is a smallest $M_{k_2}$ with $k_2 \neq k_1$ such that $f_{j_2,k_2}$ is nonzero. 
Again apply Lemma~\ref{lem:base-change} iteratively to obtain a basis for $M$ such that $f_{j_2,k_2}$ is nonzero and $f_{j_2,k}$ is zero for the $M_k$ in the block with $k \neq k_2$.
Reorder the $M_k$ in the block so that $M_{k_2}$ is second.
Also, apply Lemma~\ref{lem:base-change} iteratively to obtain a basis for $N$ such that $f_{j_2,k_2}$ is nonzero and $f_{j,k_2}$ is zero for the $N_j$ in the block with $j \neq j_2$.
Continue in the same way for the remainder of the $N_j$ in the block.
Repeat for each of the blocks.

For each $N_j$, let $M_j$ be the corresponding direct summand of $M$ obtained by the above procedure.
For the remaining $M_k$, let $N_k=0$.
\end{proof}

\begin{corollary} \label{cor:epi}
  Let $f:M \to N$ be an epimorphism between persistence modules with direct-sum decompositions into finitely many interval modules.
  Then $W_1(d_{\mu \circ \dim})(M,N) \leq \int_P \dim(\ker f) \, d\mu$.
\end{corollary}

\begin{proof}
  By Theorem~\ref{thm:matching-epi}, $M = \bigoplus_a M_a$ and $N = \bigoplus_a N_a$ where each $M_a$ is an interval module and each $N_a$ is an interval module or zero, and $f_a := p'_a f i_a$ is an epimorphism.
  Note that $M_a$ and $N_a$ have the same left ends and that $\ker f_a = M_a \setminus N_a$.
  We remark that there may be $b \neq a$ such that $p'_a f i_b$ is nonzero (see Theorem~\ref{thm:ker}).

  By the rank-nullity theorem, $\int_P \dim (\ker f) \, d\mu =
  \int_P (\dim M - \dim N) \, d\mu = \sum_a \int_P (\dim M_a - \dim N_a) \, d\mu = \sum_a \int_P \dim (M_a \setminus N_a) \, d\mu = \sum_a \int_P \dim( \ker f_a) \, d\mu = \sum_a d_{\mu \circ \dim}(M_a,N_a)$.
Therefore $W_1(d_{\mu \circ \dim})(M,N) \leq \int_P \dim (\ker f) \, d\mu$.
\end{proof}

\subsection{The $W_1$ isometry theorem} \label{sec:W1-isometry}

In this section we prove a $W_1$ isometry theorem, first in the finite case and then in the general case. The main ingredients are the induced algebraic matching theorems of the previous section.


\begin{proposition} \label{prop:metrics-same-geq}
  Let $M,N \in \cat{Vect^P_{ds}}$.
  Then $d_{\mu \circ \dim}(M,N) \leq W_1(d_{\mu \circ \dim})(M,N)$.
\end{proposition}

\begin{proof}
  We need to show that  $d_{\mu \circ \dim}(M,N) \leq \inf \norm{ \{d_{\mu \circ \dim}(M_a,N_a)\}_{a \in A}}_1$,
where the infimum is taken over all isomorphisms $M \isom \coprod_{a \in A} M_a$ and $N \isom \coprod_{a \in A} N_a$, where each $M_a$ and $N_a$ is either $0$ or an interval module.

Let $M \isom \coprod_{a \in A} M_a$ and $N \isom \coprod_{a \in A} N_a$, where each $M_a$ and $N_a$ is either $0$ or an interval module.
  For each $a \in A$, since $M_a$ and $N_a$ are either zero or an interval module, there is a zigzag $\gamma_a$ of interval modules from $M_a$ to $N_a$ of length at most two with cost $d_{\mu \circ \dim}(M_a,N_a)$.
Add identity maps to these zigzags so that they are all of the form $\cdot \to \cdot \from \cdot \to \cdot \from \cdot$.
  By taking the direct sum of the maps in these zigzags, we obtain a zigzag in $\cat{Vect^P_{ds}}$ from $M$ to $N$. Since the kernel and cokernel of a direct sum is the direct sum of the kernels and cokernels, respectively, the cost of this zigzag equals the sum of the costs of the zigzags $\gamma_a$.
  The result follows.
\end{proof}

Say that a persistence module $M$ \emph{has finite total persistence} if $\dim(M)$ is integrable, that is $\int_P \dim(M) \, d\mu < \infty$.

  \begin{remark}
    This condition can be weakened substantially using primary decomposition~\cite{thomas:thesis,Miller:2020}.
  \end{remark}

\begin{theorem}[$W_1$ isometry theorem] \label{thm:metrics-same}
  Let $M,N \in \cat{Vect^P_{ds}}$ such that each has finite total persistence.
  Then $W_1(d_{\mu \circ \dim})(M,N) = d_{\mu \circ \dim}(M,N)$.
\end{theorem}

\begin{proof}
  For simplicity, denote $d_{\mu \circ \dim}$ by $d$.
    By Proposition~\ref{prop:metrics-same-geq}, we have that
  $W_1(d)(M,N)) \geq d(M,N)$.
  So, it remains to show that 
  $W_1(d)(M,N)) \leq d(M,N)$.

  Let $\eps > 0$.
  By definition, there exists a zigzag $\gamma$ from $M$ to $N$ given by
  \[M = M_0 \xto{f_1} M_1 \xfrom{f_2} M_2 \xto{f_3} \cdots \xfrom{f_n} M_n = N\]
such that $\cost_{\mu \circ \dim}(\gamma) < d(M,N) + \frac{\eps}{2}$.
  It follows that $d(M_{i-1},M_i) < \infty$ for all $i=1,\dots,n$.
  If $M_{i-1}$ has finite total persistence and $M_i$ does not then $d(M_{i-1},M_i) = \infty$. 
  Thus we may assume that each $M_i$ has finite total persistence.

By the triangle inequality,
  \begin{equation} \label{eq:w1a}
    W_1(d)(M,N) \leq \sum_{i=1}^n W_1(d)(M_{i-1},M_i).
  \end{equation}
  Let $1 \leq i \leq n$.
  By assumption, we have $M_i \isom \bigoplus_{j=1}^{\infty} I_{i,j}$, where $I_{i,j}$ is an interval module or zero.
  Since $M_i$ has finite total persistence, we may choose $N_i$ such that
  \begin{equation} \label{eq:w1z}
    (\mu \circ \dim)(\bigoplus_{j=N_i+1}^{\infty} I_{i,j}) < \frac{\eps}{8n}.
  \end{equation}
  Let $M'_i$ denote $\bigoplus_{j=1}^{N_i} I_{i,j}$ and 
  let $M''_i$ denote $\bigoplus_{j=N_i+1}^{\infty} I_{i,j}$.
  Let $\iota_i: M'_i \to M_i$ and $\pi_i:M_i \to M'_i$ denote the canonical inclusion and projection maps.

  By the triangle inequality,
  \begin{align} \label{eq:w1b}
    \begin{split}
    W_1(d)(M_{i-1},M_i) &\leq W_1(d)(M_{i-1},M'_{i-1}) + W_1(d)(M'_{i-1},M'_i) + W_1(d)(M'_i,M_i) \\ 
    &< W_1(d)(M'_{i-1},M'_i) + \frac{\eps}{4n}.
  \end{split}
  \end{align}
  Consider the case $f_i:M_{i-1} \to M_i$.
  Let $f'_i:M'_{i-1} \to M'_i$ be given by $f'_i = \pi_i \circ f_i \circ \iota_{i-1}$.
  Since $f'_i$ factors through its image, by the triangle inequality and Corollaries \ref{cor:mono} and \ref{cor:epi},
  \begin{equation} \label{eq:w1c}
    W_1(d)(M'_{i-1},M'_i) \leq (\mu \circ \dim)(\ker f'_i) + (\mu \circ \dim)(\coker f'_i).
  \end{equation}
  Now
  \begin{align} \label{eq:w1d}
    \begin{split}
    (\mu \circ \dim)(\ker f'_i) &\leq (\mu \circ \dim)(\ker(\pi_i \circ f_i))\\
    &\leq (\mu \circ \dim)(\ker f_i) + (\mu \circ \dim)(M''_i)
  \end{split}
  \end{align}
  and
  \begin{align} \label{eq:w1e}
    \begin{split}
    (\mu \circ \dim)(\coker f'_i) &\leq (\mu \circ \dim)(\coker(f_i \circ \iota_{i-1}))\\
    &\leq (\mu \circ \dim)(\coker f_i) + (\mu \circ \dim)(M''_{i-1}).
  \end{split}
  \end{align}  
  Combining \eqref{eq:w1c}, \eqref{eq:w1d}, \eqref{eq:w1e}, and \eqref{eq:w1z} we have,
  \begin{equation} \label{eq:w1f}
    W_1(d)(M'_{i-1},M'_i) <
    (\mu \circ \dim)(\ker f_i) + (\mu \circ \dim)(\coker f_i) + \frac{\eps}{4n}. 
  \end{equation}
  The other case, $f_i:M_i \to M_{i-1}$ is similar and we obtain the same inequality as \eqref{eq:w1f}.
  Combining \eqref{eq:w1a}, \eqref{eq:w1b}, and \eqref{eq:w1f},
  we have
  \begin{equation*}
    W_1(d)(M,N) < \cost_{\mu \circ \dim}(\gamma) + \frac{\eps}{2} < d(M,N) + \eps.
  \end{equation*}
  Therefore $W_1(d)(M,N) \leq d(M,N)$.
\end{proof}

\section{Applications} \label{sec:applications}

We end by applying our distances to a few simple examples.

\subsection{Multiparameter persistence modules}

In this section we consider three examples of two-parameter persistence modules and the distances between them.

\begin{example}
  
Consider the $1$-dimensional simplicial complex $K$ at the top of Figure~\ref{fig:two-param-1}.
Let $P = \{0,1,2,3,4\}^2 \subset \Z^2$ with the usual coordinate-wise partial order and the counting measure $\mu$.
Let $X$ be the $P$-filtration of $K$ given by the vertices $a,b,c$ appearing at $(0,2), (1,1), (2,0)$, respectively, and the edge $e$ appearing at $(3,2)$ and $(2,4)$ and the edge $f$ appearing at $(2,3)$ and $(4,2)$.
See the bottom left of Figure~\ref{fig:two-param-1}.
Let $Y$ be the $P$-filtration of $K$ given by the vertices $a,b,c$ appearing at $(0,2), (1,1), (2,0)$, respectively, and the edge $e$ appearing at
$(2,3)$ and $(4,2)$
and the edge $f$ appearing at
$(3,2)$ and $(2,4)$.
See the bottom right of Figure~\ref{fig:two-param-1}.
Note that the two-parameter persistence modules $H_0(X)$ and $H_0(Y)$ have identical dimension vectors. 

\begin{figure}[ht]
  \begin{tikzpicture}
    \draw [fill, blue, thick] (0.2,0.2) circle (1pt) node(A) {} -- (0.5,0.8) circle (1pt) node(B) {} -- (0.8,0.2) circle (1pt) node(C) {};
   \node [left] at (A) {$a$};
   \node [above] at (B) {$b$};
   \node [right] at (C) {$c$};
   \node [above left] at (0.375,0.4) {$e$};
   \node [above right] at (0.675,0.4) {$f$};
 \end{tikzpicture}

 \begin{tikzpicture}
    \foreach \x in {1,2,3,4}
      \draw[shift={(\x,0)},color=black] (0,0) -- (0,5);
    \foreach \y in {1,2,3,4}
      \draw[shift={(0,\y)},color=black] (0,0) -- (5,0);
      {\draw [fill=yellow!50] (2,3) rectangle (3,4);}
      {\draw [fill=yellow!50] (3,2) rectangle (4,3);}
    \foreach \x in {0,1,2,3,4}
    \foreach \y in {2,3,4}
    \draw[shift={(\x,\y)},fill,blue] (0.2,0.2) circle (1pt);
    \foreach \x in {1,2,3,4}
    \foreach \y in {1,2,3,4}
    \draw[shift={(\x,\y)},fill,blue] (0.5,0.8) circle (1pt);
    \foreach \x in {2,3,4}
    \foreach \y in {0,1,2,3,4}
    \draw[shift={(\x,\y)},fill,blue] (0.8,0.2) circle (1pt);
    \foreach \x in {3,4}
    \foreach \y in {2,3,4}
    \draw[shift={(\x,\y)},fill,blue,thick] (0.2,0.2) -- (0.5,0.8);
    \draw[blue,thick] (2.2,4.2) -- (2.5,4.8);
    \foreach \x in {2,3,4}
    \foreach \y in {3,4}
    \draw[shift={(\x,\y)},fill,blue,thick] (0.8,0.2) -- (0.5,0.8);
    \draw[blue,thick] (4.8,2.2) -- (4.5,2.8);
\end{tikzpicture} \quad \quad
  \begin{tikzpicture}
    \foreach \x in {1,2,3,4}
      \draw[shift={(\x,0)},color=black] (0,0) -- (0,5);
    \foreach \y in {1,2,3,4}
      \draw[shift={(0,\y)},color=black] (0,0) -- (5,0);
      {\draw [fill=yellow!50] (2,3) rectangle (3,4);}
      {\draw [fill=yellow!50] (3,2) rectangle (4,3);}
    \foreach \x in {0,1,2,3,4}
    \foreach \y in {2,3,4}
    \draw[shift={(\x,\y)},fill,blue] (0.2,0.2) circle (1pt);
    \foreach \x in {1,2,3,4}
    \foreach \y in {1,2,3,4}
    \draw[shift={(\x,\y)},fill,blue] (0.5,0.8) circle (1pt);
    \foreach \x in {2,3,4}
    \foreach \y in {0,1,2,3,4}
    \draw[shift={(\x,\y)},fill,blue] (0.8,0.2) circle (1pt);
    \foreach \x in {2,3,4}
    \foreach \y in {3,4}
    \draw[shift={(\x,\y)},fill,blue,thick] (0.2,0.2) -- (0.5,0.8);
    \draw[blue,thick] (4.2,2.2) -- (4.5,2.8);
    \foreach \x in {3,4}
    \foreach \y in {2,3,4}
    \draw[shift={(\x,\y)},fill,blue,thick] (0.8,0.2) -- (0.5,0.8);
    \draw[blue,thick] (2.8,4.2) -- (2.5,4.8);
  \end{tikzpicture}
  \caption{A one dimensional simplicial complex $K$ (top) and a pair of two-parameter filtrations, $X$ (bottom left) and $Y$ (bottom right). The differences between $X$ and $Y$ are highlighted.}
  \label{fig:two-param-1}
\end{figure}

Now consider $Z := X \cap Y$ and $W := X \cup Y$.
$Z$ differs from $X$ and $Y$ in that it has no edges at the indices highlighted in Figure~\ref{fig:two-param-1}.
$W$ differs from $X$ and $Y$ in that it has both edges at the indices highlighted in Figure~\ref{fig:two-param-1}.
The inclusions $Z \xto{i} X \xto{k} W$ and $Z \xto{j} Y \xto{\ell} W$ induce two zigzags from $H_0(X)$ to $H_0(Y)$.
\begin{equation*}
  \begin{tikzcd}[row sep=tiny]
    & H_0(Z) \ar[ld, two heads, "H_0(i)"'] \ar[rd, two heads, "H_0(j)"] & \\
    H_0(X) \ar[dr, two heads, "H_0(k)"'] & & H_0(Y) \ar[dl, two heads, "H_0(\ell)"] \\
    & H_0(W) &
  \end{tikzcd}
\end{equation*}
Let $\gamma$ denote the top zigzag and let $\gamma'$ denote the bottom zigzag.
We have
$\cost_{\mu \circ \dim}(\gamma) = 
\sum_P \dim \ker H_0(i) + \sum_P \dim \ker H_0(j) = 2 + 2 = 4$ and
$\cost_{\mu \circ \dim}(\gamma') = 
\sum_P \dim \ker H_0(k) + \sum_P \dim \ker H_0(\ell) = 2 + 2 = 4$.
In either case, we have $d_{\mu \circ \dim}(H_0(X),H_0(Y)) \leq 4$.

Since $H_0(X)$ and $H_0(Y)$ have identical dimension vectors, along any zigzag from $H_0(X)$ to $H_0(Y)$ any change in the dimension vector must be later undone. Thus, $d_{\mu \circ \dim}(H_0(X),H_0(Y))$ is even.
Since $H_0(X)$ is not isomorphic to $H_0(Y)$, $d_{\mu \circ \dim}(H_0(X),H_0(Y)) \neq 0$.
It remains to show that $d_{\mu \circ \dim}(H_0(X),H_0(Y)) \neq 2$.
Since $H_0(X)$ and $H_0(Y)$ have identical dimension vectors,
this can only happen if there exists
a zigzag of length two from $H_0(X)$ to $H_0(Y)$ with middle vector space $M$
where there exists a unique $p \in P$ where $\dim M(p)$ differs from
$\dim H_0(X)(p) = \dim_1(Y)(p)$ by one and
for all $q \in P$ with $q \neq p$,
$\dim M(q) = H_0(X)(q) = \dim H_0(Y)(q)$.
However, because of the two highlighted indices in Figure~\ref{fig:two-param-1}, there is no such $M$.
Therefore $d_{\mu \circ \dim}(H_0(X),H_0(Y)) = 4$.
\end{example}

\begin{example}
  Consider the simplicial complex $K$ at the top of Figure~\ref{fig:two-param-2}.
  Let $P = [0,5]^2 \subset \R^2$ with the usual coordinate-wise partial order and the Lebesgue measure $\mu$.
  Let $t \in [0,1]$.
Let $X_t$ be the $P$-filtration of $K$ given by the vertices $a,b,c$ appearing at $(2,0), (1,1), (t,2)$, respectively, and the edge $e$ appearing at $(4,3)$ and  the edge $f$ appearing at $(3,4)$.
For $t<1$,
see the bottom left of Figure~\ref{fig:two-param-2},
and for $t=1$,
see the bottom right of Figure~\ref{fig:two-param-2}.
  
\begin{figure}[ht]
  \begin{tikzpicture}
    \draw [fill, blue, thick] (0.2,0.2) circle (1pt) node(A) {} -- (0.5,0.8) circle (1pt) node(B) {} -- (0.8,0.2) circle (1pt) node(C) {};
   \node [left] at (A) {$a$};
   \node [above] at (B) {$b$};
   \node [right] at (C) {$c$};
   \node [above left] at (0.375,0.4) {$e$};
   \node [above right] at (0.675,0.4) {$f$};
 \end{tikzpicture}

   \begin{tikzpicture}
   \path (0,0) rectangle (5.5,5);
   \draw (2,0) -- (2,5) -- (5,5) -- (5,0) -- (2,0);
   \draw (1,1) -- (1,5) -- (5,5) -- (5,1) -- (1,1);
   \draw (0.5,2) -- (0.5,5) -- (5,5) -- (5,2) -- (0.5,2);
   \draw (4,3) -- (4,5) -- (5,5) -- (5,3) -- (4,3);
   \draw (3,4) -- (3,5) -- (5,5) -- (5,4) -- (3,4);
   \node [below left] at (2,0) {$a$}; 
   \node [below left] at (1,1) {$b$};
   \node [below left] at (0.5,2) {$c$};
   \node [below left] at (4,3) {$e$};
   \node [below left] at (3,4) {$f$};
   {\draw [fill=yellow!50] (0.5,2) rectangle (1,5);}
 \end{tikzpicture} \quad \quad
 \begin{tikzpicture}
   \path (0,0) rectangle (5,5);
   \draw (2,0) -- (2,5) -- (5,5) -- (5,0) -- (2,0);
   \draw (1,1) -- (1,5) -- (5,5) -- (5,1) -- (1,1);
   \draw (1,2) -- (1,5) -- (5,5) -- (5,2) -- (1,2);
   \draw (4,3) -- (4,5) -- (5,5) -- (5,3) -- (4,3);
   \draw (3,4) -- (3,5) -- (5,5) -- (5,4) -- (3,4);
   \node [below left] at (2,0) {$a$}; 
   \node [below left] at (1,1) {$b$};
   \node [below left] at (1,2) {$c$};
   \node [below left] at (4,3) {$e$};
   \node [below left] at (3,4) {$f$};
 \end{tikzpicture}
  \caption{A one dimensional simplicial complex $K$ (top) and a pair of two-parameter filtrations, $X_t$ (bottom left) and $X_1$ (bottom right). The difference between $X_t$ and $X_1$ is highlighted.}
  \label{fig:two-param-2}
\end{figure}

Consider the two-parameter persistence modules $M_t := H_0(X_t)$ and $M_1 := H_0(X_1)$.
The inclusion $i: X_1 \incl X_t$ induces a monomorphism $H_0(i) : M_1 \incl M_t$.
Thus, by Definition~\ref{def:dmu}, $d_{\mu \circ \dim}(M_t,M_1) \leq \int_P \dim (\coker H_0(i)) \, d\mu = 3(1-t)$.
By \eqref{eq:bounds}, we also have that $d_{\mu \circ \dim}(M_t,M_1) \geq \int_P (\dim M_t - \dim M_1) \, d\mu = 3(1-t)$.
Therefore $d_{\mu \circ \dim}(M_t,M_1) = 3(1-t)$.
Note that as $t \to 1$, $d_{\mu \circ \dim}(M_t,M_1) \to 0$.
So, in this example the metric $d_{\mu \circ \dim}$ behaves continuously, as we would like. 

Now consider the metrics $W_p(d_{\mu \circ \dim})$, where $1 \leq p \leq \infty$.
Let $[x]$ denote the homology class represented by $x$.
For $t < 1$,
the persistence module $M_t$ is indecomposable.
However, $M_1 \isom A \oplus B$, where $A$ is generated by
$[a]$ and $[b]$
and $B$ is generated by
$[c] - [b]$.
By \eqref{eq:bounds}, we have that $d_{\mu \circ \dim}(M_t,A) \geq \int_P \dim M_t \, d\mu - \int_P \dim A \, d\mu \geq 39 - 29 = 10$ and $d_{\mu \circ \dim}(M_t,B) \geq \int_P \dim M_t \, d\mu - \int_P \dim B \, d\mu \geq 39 - 10 = 29$. We also have that $d_{\mu \circ \dim}(0,A) = \int_P \dim A \, d\mu = 29$, and $d_{\mu \circ \dim}(0,B) = \int_P \dim B \, d\mu = 10$.
Therefore for all $1 \leq p \leq \infty$, $W_p(d_{\mu \circ \dim})(M_t,M_1) \geq \norm{(10,10)}_p \geq 10$, even as $t \to 1$.

Since indecomposability is unstable, the metrics $W_p(d_{\mu \circ \dim})$ are also unstable.
Thus the metric $d_{\mu \circ \dim}$ seems to be a better choice for multiparameter persistence modules then the metrics $W_p(d_{\mu \circ \dim})$.
\end{example}

\begin{example}
  Consider the two-parameter persistence modules $M$, $N$, and $Q$ which are one-dimensional in the left, middle, and right subsets of the plane in Figure~\ref{fig:two-param-3}, respectively, and are zero elsewhere.
  \begin{figure}
    \centering
    \begin{tikzpicture}[scale=0.125]
      \draw [fill=yellow!50, dashed] (20,0) -- (22,0) -- (0,22) -- (0,20) -- (9,11) -- (11,11) -- (11,9) -- (20,0);
      \draw (0,22) -- (0,20) -- (9,11) -- (11,11) -- (11,9) -- (20,0) -- (22,0);
    \end{tikzpicture}
    \quad \quad
    \begin{tikzpicture}[scale=0.125]
      \draw [fill=yellow!50, dashed] (20,0) -- (22,0) -- (0,22) -- (0,20) -- (20,0);
      \draw (0,22) -- (0,20) -- (20,0) -- (22,0);
    \end{tikzpicture}
    \quad \quad
    \begin{tikzpicture}[scale=0.125]
      \draw [draw=none] (0,22) -- (0,20) -- (20,0) -- (22,0);
      \draw [fill=yellow!50, dashed] (9,11) -- (11,11) -- (11,9) -- (9,11);
      \draw (9,11) -- (11,9);
    \end{tikzpicture}
\caption{Middle: a region in the plane whose boundary is a trapezoid. Left: a subset of this region obtained by removing the triangular subregion on the right.}
    \label{fig:two-param-3}
  \end{figure}
    We have a short exact sequence $0 \to M \to N \to Q \to 0$.
    Let $\mu$ denote the Lebesgue measure on $\R^2$.
    In the path metric, $d_{\mu \circ \dim}(M,N)$ equals the area of the triangle in the right of Figure~\ref{fig:two-param-3}.
    However $W_1(d_{\mu \circ \dim})(M,N)$ equals the area of the trapezoid in the middle of Figure~\ref{fig:two-param-3}.
    Thus, $M$ and $N$ are close in the path metric and distant in the Wasserstein metric.

  Which metric is more appropriate may depend on the application. For example, 
    let $X = D^2_1 \amalg D^2_2$ be the disjoint union of two discs. Consider three bifiltrations on $X$.
    In the first, the boundary of the first disc, $\partial D^2_1$, appears on the solid lines in the middle of Figure~\ref{fig:two-param-3} and the remainder of $X$ appears on the dashed line in middle of Figure~\ref{fig:two-param-3}.
    Call this bifiltration $X_1$.
    In the second, $\partial D^2_1$  appears on the solid lines in the left of Figure~\ref{fig:two-param-3} and the remainder of $X$ appears on the dashed line in left of Figure~\ref{fig:two-param-3}.
    Call this bifiltration $X_2$.
    In the third,
    $\partial D^2_1$  appears on the left three solid lines in the left of Figure~\ref{fig:two-param-3},
    $\partial D^2_2$  appears on the right three solid lines in the left of Figure~\ref{fig:two-param-3},
    and all of $X$ appears on the dashed line in left of Figure~\ref{fig:two-param-3}.
    Call this bifiltration $X_3$.
    Then $H_1(X_1) = N$, $H_1(X_2) = M$, and $H_1(X_3) = M$.
    For $X_1$ and $X_2$, $d_{\mu \circ \dim}(M,N)$ seems to give a better answer for their proximity, but for $X_1$ and $X_3$, $W_1(d_{\mu \circ \dim})$ seems to give a better answer for their proximity.
\end{example}

\subsection{Zigzag persistence modules}

Zigzag persistence modules are linear sequences of vector spaces in which the maps are allowed to go in either direction (in a specified pattern).
For example, consider the three following three zigzag persistence modules $L$, $M$, and $N$,
\begin{equation*}
  \begin{array}{ccccccccccc}
    L & = & K & \rightarrow & K & \rightarrow & K & \leftarrow & K & \leftarrow & K \\
M & = & K & \rightarrow & K & \rightarrow & K & \leftarrow & 0 & \leftarrow & 0\\
N & = & 0 & \rightarrow & 0 & \rightarrow & K & \leftarrow & K & \leftarrow & K
  \end{array}
\end{equation*}
where in each case the maps are the identity if possible and are otherwise $0$.
These may be viewed as representations of the following quiver,
\begin{equation} \label{eq:quiver}
\bullet \rightarrow \bullet \rightarrow \bullet \leftarrow \bullet \leftarrow \bullet
\end{equation}
or modules over the corresponding path algebra, or functors from the category \eqref{eq:quiver} to the category of $K$-vector spaces.
The zigzag persistence modules $L$, $M$, and $N$, are indecomposable.
In fact, the indecomposable modules for such linear quivers are exactly the interval modules~\cite{Gabriel:1972}. However, we will show that our distances for this quiver behave differently than for the corresponding ordered quiver $\bullet \rightarrow \bullet \rightarrow \bullet \rightarrow \bullet \rightarrow \bullet$. 

As we did for persistence modules,
we consider the set of objects in the indexing category to be a subset of the integers with the counting measure $\mu$. We then have the corresponding metrics $d_{\mu \circ \dim}$ and $W_p(d_{\mu \circ \dim})$. 
However, unlike for persistence modules, the metrics $W_1(d_{\mu \circ \dim})$ and $d_{\mu \circ \dim}$ are not equal.
Indeed,
there is a surjective map $M\oplus N\rightarrow L$  whose kernel has measure one and 
so $d_{\mu \circ \dim}(M\oplus N, L)=1$.
However, for $W_1(d_{\mu \circ \dim})$ we need to match indecomposables (see Definition~\ref{def:Wp}), so $W_1(d_{\mu \circ \dim})(M \oplus N, L) = d_{\mu \circ \dim}(M,L) + d_{\mu \circ \dim}(N,0) = 2+3 = 5$. 
Which of these metrics is most appropriate will depend on the application.

\subsection*{Acknowledgments}
\label{sec:acknowledgments}

The authors would like to thank the referees whose many comments substantially improved the paper.
The first author would like to acknowledge that
this research was supported by the NSF-Simons Southeast Center for Mathematics and Biology (SCMB) through the grants National Science Foundation DMS1764406 and Simons Foundation/SFARI 594594, and that this material is based upon work supported by, or in part by, the Army Research Laboratory and the Army Research Office under contract/Grant No. W911NF-18-1-0307.




\bibliographystyle{hplain}
\bibliography{phalg}

\begin{thebibliography}{10}

\bibitem{Bauer:2020}
Ulrich Bauer, Magnus~B. Botnan, Steffen Oppermann, and Johan Steen.
\newblock Cotorsion torsion triples and the representation theory of filtered
  hierarchical clustering.
\newblock {\em Adv. Math.}, 369:107171, 2020.

\bibitem{bauerLesnick}
Ulrich Bauer and Michael Lesnick.
\newblock Induced matchings and the algebraic stability of persistence
  barcodes.
\newblock {\em J. Comput. Geom.}, 6(2):162--191, 2015.

\bibitem{Blumberg:2017}
Andrew~J. Blumberg and Michael Lesnick.
\newblock Universality of the homotopy interleaving distance.
\newblock 05 2017, arXiv:1705.01690 [math.AT].

\bibitem{Botnan:2018}
Magnus~Bakke Botnan and William Crawley-Boevey.
\newblock Decomposition of persistence modules.
\newblock {\em Proc. Amer. Math. Soc.}, 148(11):4581--4596, 2020.

\bibitem{Botnan:2020a}
Magnus~Bakke Botnan, Justin Curry, and Elizabeth Munch.
\newblock A relative theory of interleavings.
\newblock 04 2020, arXiv:2004.14286 [math.CT].

\bibitem{bdss:1}
Peter Bubenik, Vin de~Silva, and Jonathan Scott.
\newblock Metrics for {G}eneralized {P}ersistence {M}odules.
\newblock {\em Found. Comput. Math.}, 15(6):1501--1531, 2015.

\bibitem{bdss:2}
Peter Bubenik, Vin de~Silva, and Jonathan Scott.
\newblock Interleaving and {G}romov-{H}ausdorff distance.
\newblock 2017, arXiv:1707.06288 [math.CT].

\bibitem{be:universality}
Peter Bubenik and Alex Elchesen.
\newblock Universality of persistence diagrams and the bottleneck and
  {W}asserstein distances.
\newblock {\em Computational Geometry}, 105-106:101882, 2022.

\bibitem{Bubenik:2020a}
Peter Bubenik and Alex Elchesen.
\newblock Virtual persistence diagrams, signed measures, {W}asserstein
  distances, and {B}anach spaces.
\newblock {\em Journal of Applied and Computational Topology}, 2022,
  doi:10.1007/s41468-022-00091-9.

\bibitem{BubenikMilicevic}
Peter Bubenik and Nikola Mili{\'c}evi{\'c}.
\newblock Homological algebra for persistence modules.
\newblock {\em Foundations of Computational Mathematics}, 21(5):1233--1278,
  2021.

\bibitem{bubenikScott:1}
Peter Bubenik and Jonathan~A. Scott.
\newblock Categorification of persistent homology.
\newblock {\em Discrete Comput. Geom.}, 51(3):600--627, 2014.

\bibitem{Bubenik:2018}
Peter Bubenik and Tane Vergili.
\newblock Topological spaces of persistence modules and their properties.
\newblock {\em J. Appl. Comput. Topol.}, 2(3-4):233--269, 2018.

\bibitem{Bucur:1968}
Ion Bucur and Aristide Deleanu.
\newblock {\em Introduction to the theory of categories and functors}.
\newblock With the collaboration of Peter J. Hilton and Nicolae Popescu. Pure
  and Applied Mathematics, Vol. XIX. Interscience Publication John Wiley \&
  Sons, Ltd., London-New York-Sydney, 1968.

\bibitem{ccsggo:interleaving}
Fr\'{e}d\'{e}ric Chazal, David Cohen-Steiner, Marc Glisse, Leonidas~J. Guibas,
  and Steve~Y. Oudot.
\newblock Proximity of persistence modules and their diagrams.
\newblock In {\em Proceedings of the 25th annual symposium on Computational
  geometry}, SCG '09, pages 237--246, New York, NY, USA, 2009. ACM.

\bibitem{cseh:stability}
David Cohen-Steiner, Herbert Edelsbrunner, and John Harer.
\newblock Stability of persistence diagrams.
\newblock {\em Discrete Comput. Geom.}, 37(1):103--120, 2007.

\bibitem{csehm:lipschitz}
David Cohen-Steiner, Herbert Edelsbrunner, John Harer, and Yuriy Mileyko.
\newblock Lipschitz functions have {$L_p$}-stable persistence.
\newblock {\em Found. Comput. Math.}, 10(2):127--139, 2010.

\bibitem{Crawley-Boevey:2015}
William Crawley-Boevey.
\newblock Decomposition of pointwise finite-dimensional persistence modules.
\newblock {\em J. Algebra Appl.}, 14(5):1550066, 8, 2015.

\bibitem{deSilva:2017}
V.~de~Silva, E.~Munch, and A.~Stefanou.
\newblock Theory of interleavings on categories with a flow.
\newblock {\em Theory Appl. Categ.}, 33:Paper No. 21, 583--607, 2018.

\bibitem{deSilva:2016}
Vin de~Silva, Elizabeth Munch, and Amit Patel.
\newblock Categorified {R}eeb graphs.
\newblock {\em Discrete Comput. Geom.}, 55(4):854--906, 2016.

\bibitem{Divol:2019}
Vincent Divol and Th\'{e}o Lacombe.
\newblock Understanding the topology and the geometry of the space of
  persistence diagrams via optimal partial transport.
\newblock {\em J. Appl. Comput. Topol.}, 5(1):1--53, 2021.

\bibitem{Elchesen:2018}
Alexander Elchesen and Facundo M\'{e}moli.
\newblock The reflection distance between zigzag persistence modules.
\newblock {\em J. Appl. Comput. Topol.}, 3(3):185--219, 2019.

\bibitem{Gabriel:1972}
Peter Gabriel.
\newblock Unzerlegbare {D}arstellungen. {I}.
\newblock {\em Manuscripta Math.}, 6:71--103; correction, ibid. 6 (1972), 309,
  1972.

\bibitem{Giunti:2021c}
Barbara Giunti, John~S. Nolan, Nina Otter, and Lukas Waas.
\newblock Amplitudes on abelian categories.
\newblock 07 2021, arXiv:2107.09036 [math.AT].

\bibitem{Harker:2019}
Shaun Harker, Miroslav Kram\'{a}r, Rachel Levanger, and Konstantin Mischaikow.
\newblock A comparison framework for interleaved persistence modules.
\newblock {\em J. Appl. Comput. Topol.}, 3(1-2):85--118, 2019.

\bibitem{Harrington:2017}
Heather~A. Harrington, Nina Otter, Hal Schenck, and Ulrike Tillmann.
\newblock Stratifying {M}ultiparameter {P}ersistent {H}omology.
\newblock {\em SIAM J. Appl. Algebra Geom.}, 3(3):439--471, 2019.

\bibitem{Krause:2015}
Henning Krause.
\newblock Krull-{S}chmidt categories and projective covers.
\newblock {\em Expo. Math.}, 33(4):535--549, 2015.

\bibitem{Lesnick:2011}
Michael Lesnick.
\newblock The theory of the interleaving distance on multidimensional
  persistence modules.
\newblock {\em Found. Comput. Math.}, 15(3):613--650, 2015.

\bibitem{lesnickWright:rivet}
Michael Lesnick and Matthew Wright.
\newblock Interactive visualization of 2-d persistence modules.
\newblock arXiv:1512.00180 [math.AT], 2015.

\bibitem{McCleary:2018}
Alex McCleary and Amit Patel.
\newblock Bottleneck stability for generalized persistence diagrams.
\newblock {\em Proceedings of the American Mathematical Society}, page~1, Nov
  2019.

\bibitem{Miller:2019}
Ezra Miller.
\newblock Modules over posets: commutative and homological algebra.
\newblock 08 2019.

\bibitem{Miller:2020c}
Ezra Miller.
\newblock Essential graded algebra over polynomial rings with real exponents.
\newblock 08 2020.

\bibitem{Miller:2020}
Ezra Miller.
\newblock Primary decomposition over partially ordered groups.
\newblock 08 2020.

\bibitem{Morozov:2013}
Dmitriy Morozov, Kenes Beketayev, and Gunther~H. Weber.
\newblock Interleaving distance between merge trees.
\newblock In {\em Proceedings of TopoInVis}, 2013.

\bibitem{Munch:2018}
Elizabeth Munch and Anastasios Stefanou.
\newblock The {$\ell^\infty$}-cophenetic metric for phylogenetic trees as an
  interleaving distance.
\newblock In {\em Research in data science}, volume~17 of {\em Assoc. Women
  Math. Ser.}, pages 109--127. Springer, Cham, 2019.

\bibitem{Pareigis:1970}
Bodo Pareigis.
\newblock {\em Categories and functors}.
\newblock Translated from the German. Pure and Applied Mathematics, Vol. 39.
  Academic Press, New York-London, 1970.

\bibitem{Patel:2018}
Amit Patel.
\newblock Generalized persistence diagrams.
\newblock {\em J. Appl. Comput. Topol.}, 1(3-4):397--419, 2018.

\bibitem{Popescu:1973}
N.~Popescu.
\newblock {\em Abelian categories with applications to rings and modules}.
\newblock Academic Press, London-New York, 1973.

\bibitem{Scolamiero:2017}
Martina Scolamiero, Wojciech Chach\'olski, Anders Lundman, Ryan Ramanujam, and
  Sebastian \"Oberg.
\newblock Multidimensional persistence and noise.
\newblock {\em Found. Comput. Math.}, 17(6):1367--1406, 2017.

\bibitem{SkrabaTurner}
Primoz Skraba and Katharine Turner.
\newblock Wasserstein stability for persistence diagrams.
\newblock 06 2020, arXiv:2006.16824 [math.AT].

\bibitem{Stenstrom:1975}
Bo~Stenstr\"{o}m.
\newblock {\em Rings of quotients}.
\newblock Springer-Verlag, New York-Heidelberg, 1975.

\bibitem{thomas:thesis}
Ashleigh Thomas.
\newblock {\em Invariants and Metrics for Multiparameter Persistent Homology}.
\newblock PhD thesis, Duke University, 2019.

\end{thebibliography}

\end{document}